\documentclass[10pt]{amsart}
\usepackage{times}
\usepackage{amssymb}
\usepackage{amsmath}
\usepackage{amsfonts}
\usepackage{latexsym}
\usepackage{amsthm}
\newtheorem{thm}{Theorem}[section]
\newtheorem{lemma}[thm]{Lemma}
\newtheorem{cor}[thm]{Corollary}
\newtheorem{prop}[thm]{Proposition}
\newtheorem{otherth}{\bf Theorem}

\newtheorem{otherl}{\bf Lemma}

\def\Bn{\mathbb{B}_n}
\def\Cn{\mathbb{C}^n}
\def\C{\mathbb{C}}

\def\a{\alpha}
\def\b{\beta}
\def\g{\gamma}

\def\e{\varepsilon}
\def\l{\lambda}
\def\p{\varphi}

\def\p{\varphi}

\numberwithin{equation}{section}

\begin{document}

\title[Carleson Measures and Toeplitz operators]
{Carleson Measures and Toeplitz operators for weighted Bergman spaces on the unit ball}

\author[J. Pau]
{Jordi Pau}
\address{
Jordi Pau\\
Departament de Matem\`atica Aplicada i Analisi,
Universitat de Barcelona,
08007 Barcelona,
Spain}
\email{jordi.pau@ub.edu}

\author[R. Zhao]
{Ruhan Zhao}
\address{
Ruhan Zhao\\
Department of Mathematics,
SUNY Brockport,
Brockport, NY 14420,
USA}
\email{rzhao@brockport.edu}

% TOP MATTER


\thanks{This work was completed while the second named author visited the University of Barcelona in 2013.
He thanks the support given by the IMUB and the SGR grant $2009$SGR $420$ (Generalitat de
Catalunya) from which the first author is also partially supported. The first author is also
 supported by DGICYT grant MTM$2011$-$27932$-$C02$-$01$
(MCyT/MEC)}

\begin{abstract}
Some new characterizations on Carleson measures
for weighted Bergman spaces on the unit ball
involving product of functions are obtained.
For these we characterize bounded and compact
Toeplitz operators between weighted Bergman spaces.
The above results are applied to characterize bounded and compact
extended Ces\`aro operators and pointwise multiplication operators.
The results are new even in the case of the unit disk.

\end{abstract}

%\subjclass{Primary 30D35; Secondary }
%\keywords{Carleson measures, weighted Bergman spaces, $F(p,q,s)$ spaces,
%Riemann-Stieltjes operators,
%pointwise multiplication operators}

\maketitle

\section{Introduction}
\label{intro}
Let $\Cn$ denote the Euclidean space of complex dimension $n$.
For any two points $z=(z_ 1,\dots,z_ n)$ and $w=(w_ 1,\dots,w_ n)$ in $\Cn$ we write
$\langle z,w\rangle =z_ 1\bar{w}_ 1+\dots +z_ n \bar{w}_ n,$
 and
 $|z|=\sqrt{\langle z,z\rangle}=\sqrt{|z_ 1|^2+\dots +|z_ n|^2}.$
Let $\Bn=\{z\in \C^n:|z|<1\}$ be the unit ball in $\C^n$.
Let $H(\Bn)$ be the space of all holomorphic functions on the
unit ball $\Bn$.
Let $dv$ be the normalized volume measure on $\Bn$ such that
$v(\Bn)=1$.
For $0<p<\infty$ and $-1<\a<\infty$, let
$L^{p,\a}:=L^{p}(\Bn,\,dv_{\a})$ denote the weighted Lebesgue spaces
which contain measurable functions $f$ on $\Bn$
such that
$$
\|f\|_{p,\a}=\left (\int_{\Bn}|f(z)|^p\,dv_{\a}(z)\right )^{1/p}<\infty,
$$
where $dv_{\a}(z)=c_{\a}(1-|z|^2)^{\a}\,dv(z)$, and $c_{\a}$ is
the normalized constant such that $v_{\a}(\Bn)=1$.
We also denote by $A^p_{\a}=L^{p}(\Bn,\,dv_{\a})\cap H(\Bn)$,
the weighted Bergman space on $\Bn$, with the same norm as above.
If $\a=0$, we simply write them as $L^p(\Bn,\,dv)$ and $A^p$ respectively
and $\|f\|_p$ for the norm of $f$ in these spaces.

Let $\mu$ be a positive Borel measure on $\Bn$. For $\lambda>0$ and $\a>-1$
we say $\mu$ is a \emph{$(\lambda,\a)$-Bergman Carleson measure} if
for any two positive numbers $p$ and $q$ with $q/p=\lambda$
there is a positive constant $C>0$ such that
$$
\int_{\Bn} |f(z)|^q\,d\mu(z)\le C\|f\|_{p,\a}^q
$$
for any $f\in A^p_{\a}$. We also denote by
$$
\|\mu\|_{\l,\a}=\sup_{f\in A^p_{\a}, \|f\|_{p,\a}\le 1}\int_{\Bn} |f(z)|^q\,d\mu(z).
$$

The concept of Carleson measures was first introduced by L. Carleson in order
to study interpolating sequences and the corona problem \cite{Car0,Car1}
on the algebra $H^\infty$ of all bounded analytic functions on the unit disk.
It quickly became a powerful tool for the study of function spaces and
operators acting on them. The Bergman Carleson
measures were first studied by Hastings \cite{H},
and further pursued by Oleinik \cite{Ol}, Luecking \cite{l0,l1},
Cima-Wogen \cite{CW}, and many others.

In this paper we will give new characterizations for
$(\l,\g)$-Bergman Carleson measures and vanishing
$(\l,\g)$-Bergman Carleson measures (defined in Section 4) on the unit ball $\Bn$
by using the products of functions in weighted Bergman spaces.
For proving these results, we have to characterize bounded and compact
Toeplitz operators between weighted Bergman spaces, which is
of independent interest.
Our results will be applied to study boundedness and compactness of
extended Ces\`aro operators and pointwise multiplication operators
from weighted Bergman spaces to a general family of function spaces.

\begin{thm}\label{carleson}
Let $\mu$ be a positive Borel measure on $\Bn$.
For any integer $k\ge 1$ and $i=1,2,...,k$,
let $0<p_i,q_i<\infty$ and $-1<\a_i<\infty$.
Let
\begin{equation}\label{Eqlg}
\l=\sum_{i=1}^k\frac{q_i}{p_i}; \qquad
\g=\frac1{\l}\sum_{i=1}^k\frac{\a _iq_i}{p_i}.
\end{equation}
Then $\mu$ is a $(\l,\g)$-Bergman Carleson measure
if and only if there is a constant $C>0$ such that
for any $f_i\in A^{p_i}_{\a_i}$, $i=1,2,..., k$,
\begin{equation}\label{eq1}
\int_{\Bn}\prod_{i=1}^k|f_i(z)|^{q_i}\,d\mu(z)
\le C\prod_{i=1}^k\|f_i\|_{p_i,\a_i}^{q_i}.
\end{equation}
\end{thm}

A similar result for Hardy spaces on the unit disk
was given by the second author in \cite{zha3}.
Due to lack of Riesz factorization theorem for weighted Bergman spaces,
the proof of the above theorem will be quite different and involved.
For the proof of one implication in the case $0<\l<1$,
a description of bounded Toeplitz operators between different Bergman spaces is needed,
a result that can be of independent interest that we are going to state next.
Given $\b>-1$ and a positive Borel measure $\mu$ on $\Bn$,
define the Toeplitz operator $T_{\mu}^{\b}$ as follows:
$$
T_{\mu}^{\b}f(z)=\int_{\Bn}\frac{f(w)}{(1-\langle z, w \rangle )^{n+1+\b}}\,d\mu(w),\qquad z\in \Bn.
$$

\begin{thm}\label{toeplitz1} Let $0<p_1, p_2<\infty$, $-1<\a_1,\a_2<\infty$.
Suppose that
\begin{equation*}
n+1+\beta> n \max \big (1,\frac{1}{p_ i}\big )+\frac{1+\alpha_ i}{p_ i},\qquad i=1,2.
\end{equation*}
Let
$$
\l=1+\frac1{p_1}-\frac1{p_2},\quad
\g=\frac{1}{\l}\left(\b+\frac{\a_1}{p_1}-\frac{\a_2}{p_2}\right),
$$
Let $\mu$ be a positive Borel measure on $\Bn$. Then the following statements are equivalent:
\begin{itemize}
\item[(i)] $T_{\mu}^{\b}$ is bounded from $A^{p_1}_{\a_1}$ to $A^{p_2}_{\a_2}$.
\item[(ii)] The measure $\mu$ is a $(\l,\g)$-Bergman Carleson measure.
%\item[(iii)] The measure $\mu$ satisfies
%\begin{equation*}
%\sup_{a\in D}\int_{\Bn} |\p'_a(z)|^{s}\,d\mu(z) <\infty.
%\end{equation*}
\end{itemize}
Moreover, one has
$$\|T_{\mu}^{\b}\|_{A^{p_1}_{\a_1} \rightarrow A^{p_2}_{\a_2}} \asymp \|\mu\|_{\l,\g}.$$
\end{thm}
\mbox{}
\\
\emph{Remark:}
In Theorem \ref{toeplitz1}, the condition
\begin{equation}\label{B-C1}
n+1+\beta> n \max \big (1,\frac{1}{p_ 1}\big )+\frac{1+\alpha_ 1}{p_ 1}
\end{equation}
 is used to prove the implication (i) implies (ii), while the condition
 \begin{equation}\label{B-C2}
n+1+\beta> n \max \big (1,\frac{1}{p_ 2}\big )+\frac{1+\alpha_ 2}{p_ 2}
\end{equation}
is needed to prove the other implication (ii) implies (i). Moreover, when $p_ 1\ge 1$, then
\eqref{B-C1} reduces to
$(1+\beta) p_ 1 >1+\alpha_ 1,$
and by Theorem 2.11 of Zhu's book \cite{ZhuBn}, this is equivalent to the fact that $P_{\beta}$ is a bounded projection from $L^{p_ 1}(\Bn,dv_{\alpha _ 1})$ onto $A^{p_ 1}_{\alpha _ 1}$. Here, the projection $P_{\beta}$ is defined as
$$P_{\beta} f(z)=\int_{\Bn} \frac{ f(w)\,dv_{\beta}(w)}{(1-\langle z,w \rangle )^{n+1+\beta}}.$$
In a similar way, when $p_ 2\ge 1$, the condition \eqref{B-C2} is equivalent to the fact that $P_{\beta}$ is a bounded projection from $L^{p_ 2}(\Bn,dv_{\alpha _ 2})$ onto $A^{p_ 2}_{\alpha _ 2}$. \\

The paper is organized as follows. In Section 2 we recall some notations
and preliminary results which will be used later.
Section 3 is devoted to the proofs of our main results,
Theorem~\ref{carleson} and Theorem~\ref{toeplitz1}.
In Section 4 we give similar characterizations for vanishing
$(\lambda,\gamma)$-Bergman Carleson measures.
In Section 5 we apply Theorem~\ref{carleson} to
characterize bounded extended Ces\`aro operators
and pointwise multiplication operators from
weighted Bergman spaces into a general family of function spaces.

In the following, the notation $A\lesssim B$ means
that there is a positive constant $C$ such that $A\leq CB$,
and the notation $A\asymp B$ means that both $A\lesssim B$ and $B\lesssim A$ hold.

\section{Preliminaries}
\label{b-carleson}

In this section we introduce some notations and recall some well known results
that will be used throughout the paper.

For any $a\in \Bn$ with $a\neq 0$, we denote by
$\p_a(z)$ the M\"obius transformation on $\Bn$
that interchanges the points $0$ and $a$.
%It is known that, for any $z\in\Bn$
%$$\p_a(z)=\frac{a-P_a(z)-s_aQ_a(z)}{1-\langle z,a\rangle},$$
%where $s_a =1-|a|^2$ , $P_a$ is the orthogonal projection
%from $\C^n$ onto the one dimensional subspace $[a]$ generated by $a$,
%and $Q_a$ is the orthogonal projection from $\C^n$ onto the
%orthogonal complement of $[a]$. When $a=0$, $\p_a(z)=-z$.
It is known that $\p_a$ satisfies the following
properties: $\p_a\circ\p_a(z)=z$, and
\begin{equation}\label{eq-pa}
1-|\p_a(z)|^2=\frac{(1-|a|^2)(1-|z|^2)}{|1-\langle z,a\rangle|^2}, \qquad a,z\in \Bn.
\end{equation}
For $z,w\in\Bn$, the \emph{pseudo-hyperbolic distance} between $z$ and $w$
is defined by
$$
\rho(z,w)=|\p_z(w)|,
$$
and the \emph{hyperbolic distance} on $\Bn$ between $z$ and $w$
induced by the Bergman metric is given by
$$
\beta(z,w)=\tanh\,\rho(z,w)=\frac12\log\frac{1+|\p_z(w)|}{1-|\p_z(w)|}.
$$
For $z\in\Bn$ and $r>0$, the \emph{Bergman metric ball} at $z$ is given by
$$
D(z,r)=\big \{w\in\Bn:\,\beta(z,w)<r \big \}.
$$
It is known that, for a fixed $r>0$, the weighted volume
$$
v_{\a}(D(z,r))\asymp (1-|z|^2)^{n+1+\a}.
$$
We refer to \cite{ZhuBn} for the above facts.

We cite two results for Bergman Carleson measures which justify
the fact that a Bergman Carleson measure depends only on
$\a$ and the ratio $\lambda=q/p$.
The first result was obtained by several authors and can be found,
for example, in \cite[Theorem 50]{zz} and the references there.

\begin{otherth}\label{TA}
For a positive Borel measure $\mu$ on $\Bn$,
$0<p\le q<\infty$, and $-1<\a<\infty$,
the following statements are equivalent:
\begin{itemize}
\item[(i)] There is a constant $C_1>0$ such that for any $f\in A^p_{\a}$
$$
\int_{\Bn}|f(z)|^q\,d\mu(z)\le C_1\|f\|_{p,\a}^q.
$$
\item[(ii)] There is a constant $C_2>0$ such that,
for any real number $r$ with $0<r<1$ and any $z\in \Bn$,
$$
\mu(D(z,r))\le C_2(1-|z|^2)^{(n+1+\a)q/p}.
$$
\item[(iii)] There is a constant $C_3>0$ such that,
for some (every) $t>0$,
$$
\sup_{a\in\Bn}\int_{\Bn}\frac{(1-|a|^2)^t}{|1-\langle z, a \rangle |^{[{(n+1+\a)q/p}]+t}}\,d\mu(z)\le C_3.
$$
\end{itemize}
Furthermore, the constants $C_1$, $C_2$ and $C_3$ are all comparable
to $\|\mu\|_{\l,\a}$ with $\l=q/p$.
\end{otherth}

\textbf{Remark.}
Let $\lambda=q/p$. Then the above result states that
a positive Borel measure $\mu$ on $\Bn$
is a $(\lambda,\a)$-Bergman Carleson measure if and only if
$$
\sup_{s\in\Bn}\int_{\Bn}\frac{(1-|a|^2)^t}
{|1-\langle z, a \rangle |^{(n+1+\a)\lambda+t}}\,d\mu(z)<\infty.
$$
 for some (every) $t>0$.

For the case $0<q<p<\infty$, we need a well-known
result on decomposition of the unit ball $\Bn$.
A sequence $\{a_k\}$ of points in $\Bn$ is called
a \emph{separated sequence} (in the Bergman metric)
if there exists a positive constant $\delta>0$ such that
$\beta(z_i,z_j)>\delta$ for any $i\neq j$.
The following result is Theorem 2.23 in \cite{ZhuBn}.

\begin{otherl}\label{covering}
There exists a positive integer $N$ such that for any $0<r<1$ we can
find a sequence $\{a_k\}$ in $\Bn$ with the following properties:
\begin{itemize}
\item[(i)] $\Bn=\cup_{k}D(a_k,r)$.
\item[(ii)] The sets $D(a_k,r/4)$ are mutually disjoint.
\item[(iii)] Each point $z\in\Bn$ belongs to at most $N$ of the sets $D(a_k,4r)$.
\end{itemize}
\end{otherl}

Any sequence $\{a_k\}$ satisfying the conditions of the above lemma is called
a \emph{lattice} (or an $r$-$lattice$ if one wants to stress the dependence on $r$)
in the Bergman metric. Obviously any $r$-lattice is separated.
For convenience, we will denote by $D_k=D(a_k,r)$ and $\tilde{D}_k=D(a_k,4r)$.
Then Lemma~\ref{covering} says that $\Bn=\cup_{k=1}^{\infty}D_k$
and there is an positive integer $N$ such that every point $z$ in $\Bn$
belongs to at most $N$ of sets $\tilde{D}_k$.

The following result is essentially due to D. Luecking (\cite{l2} and \cite{l3}),
for the case $\a=0$ (note that the discrete form (iii) is actually given in
Luecking's proof).
For $-1<\a<\infty$, the result can be similarly proved as in \cite{l3}.
The condition in part (iv) first appeared in \cite{CKY} (see also \cite[Theorem 54]{zz}),
where it was used for the embedding of harmonic Bergman spaces into Lebesgue spaces.

\begin{otherth}\label{TB}
 For a positive Borel measure $\mu$ on $\Bn$,
$0<q<p<\infty$ and $-1<\a<\infty$,
the following statements are equivalent:
\begin{itemize}
\item[(i)] There is a constant $C_1>0$ such that for any $f\in A^p_{\a}$
$$
\int_{\Bn}|f(z)|^q\,d\mu(z)\le C_1\|f\|_{p,\a}^q.
$$
\item[(ii)] The function
$$
\widehat{\mu}_ r(z):=\frac{\mu(D(z,r))}{(1-|z|^2)^{n+1+\a}}
$$
is in $L^{p/(p-q),\a}$ for any (some)
fixed $r\in(0,1)$.
\item[(iii)] For any $r$-lattice $\{a_k\}$ and $D_k$ as in Lemma~\ref{covering},
the sequence
$$
\{\mu_ k\}:=\left\{\frac{\mu(D_k)}{(1-|a_k|^2)^{(n+1+\a)\frac{q}{p}}}\right \}
$$
belongs to $\ell^{p/(p-q)}$ for any (some)
fixed $r\in(0,1)$.
\item[(iv)] For any $s>0$, the Berezin-type transform $B_{s,\a} (\mu)$ belongs to $L^{p/(p-q),\a}$.
\end{itemize}
Furthermore, with $\lambda=q/p$, one has
$$\|\widehat{\mu}_ r\|_{\frac{p}{p-q},\a} \asymp \|\{\mu_ k\}\|_{\ell^{p/(p-q)}}\asymp \|B_{s,\a} (\mu)\|_{\frac{p}{p-q},\a} \asymp \|\mu\|_{\lambda,\alpha}.$$
\end{otherth}
\mbox{}
\\
Here, for a positive measure $\nu$, the Berezin-type transform $B_{s,\alpha}(\nu)$ is
$$
B_{s,\alpha}(\nu)(z)=\int_{\Bn} \frac{(1-|z|^2)^s}{|1-\langle z,w \rangle |^{n+1+s+\alpha}} \,d\nu(w).
$$
As a consequence of the previous stated result, for $0<\l<1$,
a positive Borel measure $\mu$ on $\Bn$ is a
$(\lambda,\a)$-Bergman Carleson measure if and only if
$$
\mu(D(z,r))(1-|z|^2)^{-n-1-\a}\in L^{1/(1-\lambda),\a},
$$
 or
$$
\{\mu(D_k)(1-|a_k|^2)^{-(n+1+\a)\lambda}\}\in \ell^{1/(1-\lambda)}
$$
for any (some) fixed $r\in(0,1)$.
\\

The following integral estimate (see \cite[Theorem 1.12]{ZhuBn}) has become
indispensable in this area of analysis, and will be used
several times in this paper.

\begin{otherl}\label{Ict}
Suppose $z\in\Bn$, $c>0$  and $t>-1$. The integral
$$
I_{c,t}(z)=\int_{\Bn}\frac{(1-|w|^2)^t}{|1-\langle z,w\rangle |^{n+1+t+c}}\,dv(w)
$$
is comparable to $(1-|z|^2)^{-c}$.
\end{otherl}
We also need a well known variant of the previous lemma.

\begin{otherl}\label{l2}
Let $\{z_ k\}$ be a separated sequence in $\Bn$, and let $n<t<s$.
Then
$$
\sum_{k=1}^{\infty}\frac{(1-|z_ k|^2)^t}{|1-\langle z,z_ k \rangle |^s}\le C\,
(1-|z|^2)^{t-s},\qquad z\in \Bn.
$$
\end{otherl}

Lemma~\ref{l2} can be deduced from Lemma \ref{Ict} after noticing that,
if a sequence $\{z_ k\}$ is separated, then there is a constant
$r>0$ such that the Bergman metric balls $D(z_ k,r)$ are pairwise
disjoints. With all these preparations now we are ready to prove the main results.

\section{Proofs}
We need first the following lemma.
\begin{lemma}\label{lem1} Let $-1<\a<\infty$.
For $i=1,2,\dots, k$, let $0<p_i, q_i<\infty$, and let $f_i\in A^{p_i/q_i}_{\a_i}$.
Let
$$
\l=\sum_{i=1}^k\frac{q_i}{p_i}; \qquad
\g=\frac1{\l}\sum_{i=1}^k\frac{\a _iq_i}{p_i}.
$$
Then $\prod_{i=1}^kf_i\in A^{1/\l}_\g,$
and
$$
\left\|\prod_{i=1}^kf_i\right\|_{1/\l,\g}\lesssim \prod_{i=1}^k\|f_i\|_{p_i/q_i,\a_i}.
$$
\end{lemma}

\begin{proof}
Let $f_i\in A^{p_i/q_i}_{\a_i}$ ($i=1,2,\dots, k$). Since $p_i\l/q_i>1$ for any $i=1,2,\dots,k$,
we can apply H\"older's inequality to obtain
\begin{eqnarray*}
\left\|\prod_{i=1}^k f_i\right\|_{1/\l,\g}
&=&\left(c_{\gamma}\int_{\Bn}\prod_{i=1}^k |f_i(z)|^{1/\l}(1-|z|^2)^{\g}\,dv(z)\right)^{\l}
\\
&\lesssim&\prod_{i=1}^k \left(\int_{\Bn}|f_ i(z)|^{\frac 1{\l}\frac{p_i\l}{q_i}}(1-|z|^2)^{\frac{q_i\a_i}{p_i\l}\frac{p_i\l}{q_i}}
\,dv(z)\right)^{\frac{q_i}{p_i}}
\\
&=& \prod_{i=1}^k \left(\int_{\Bn}|f_i(z)|^{\frac{p_i}{q_i}}(1-|z|^2)^{\a_i}\,dv(z)\right)^{\frac{q_i}{p_i}}
\\
&\lesssim &\prod_{i=1}^k\|f_i\|_{p_i/q_i,\a_i}.
\end{eqnarray*}
The result is proved.
\end{proof}

\begin{prop}\label{Car-P1}
Let $\mu$ be a positive Borel measure on $\Bn$.
For any integer $k\ge 1$ and $i=1,2,\dots,k$,
let $0<p_i,q_i<\infty$ and $-1<\a_i<\infty$, and let $\l$ and $\g$ be as in \eqref{Eqlg}. If
$\mu$ is a $(\l,\g)$-Bergman Carleson measure then \eqref{eq1} holds.
\end{prop}

\begin{proof} If $k=1$ then the result is just the definition.
Let us now assume $k\ge2$.
Let $h_i\in A^{p_i/q_i,\a_i}$, $i=1,2,\dots,k$. By Lemma \ref{lem1},
 $\prod_{i=1}^k h_i\in A^{1/\l}_\g,$
and
$$
\left\|\prod_{i=1}^k h_i\right\|_{1/\l,\g}\lesssim \prod_{i=1}^k\|h_i\|_{p_i/q_i,\a_i}.
$$
Since $\mu$ is a $(\l,\g)$-Bergman Carleson measure,
\begin{equation}\label{eq2}
\int_{\Bn} \left|\prod_{i=1}^k h_i(z)\right|\,d\mu(z)
\le C\left\|\prod_{i=1}^k h_i\right\|_{1/\l,\g}
\le C\prod_{i=1}^k\|h_i\|_{p_i/q_i,\a_i}.
\end{equation}
Let
$$
d\mu_1=\left(\prod_{i=2}^k|h_i|\,d\mu\right)\bigg/\left(\prod_{i=2}^k\|h_i\|_{p_i/q_i,\a_i}\right).
$$
Then \eqref{eq2} is equivalent to
$$
\int_{\Bn}|h_1(z)|\,d\mu_1(z)\le C\|h_1\|_{p_1/q_1,\a_1}.
$$
Thus $\mu_1$ is a $(q_1/p_1,\a_1)$-Bergman Carleson measure.
Thus for any $f_1\in A^{p_1}_{\a_1}$,
$$
\int_{\Bn}|f_1(z)|^{q_1}\,d\mu_1(z)\le C\|f_1\|_{p_1,\a_1}^{q_1},
$$
which is the same as
\begin{equation}\label{eq3}
\int_{\Bn}|f_1(z)|^{q_1}\prod_{i=2}^k|h_i(z)|\,d\mu(z)
\le C\|f_1\|_{p_1,\a_1}^{q_1}\prod_{i=2}^k\|h_i\|_{p_i/q_i,\a_i}.
\end{equation}
Let
$$
d\mu_2=\left(|f_1|^{q_1}\prod_{i=3}^k|h_i|\,d\mu\right)
\bigg/\left(\|f_1\|_{p_1,\a_1}^{q_1}\prod_{i=3}^k\|h_i\|_{p_i/q_i,\a_i}\right).
$$
Then (\ref{eq3}) is the same as
$$
\int_{\Bn}|h_2(z)|\,d\mu_2(z)\le C\|h_2\|_{p_2/q_2,\a_2}.
$$
Thus $\mu_2$ is a $(q_2/p_2,\a_2)$-Bergman Carleson measure.
Thus for any $f_2\in A^{p_2}_{\a_2}$,
$$
\int_{\Bn}|f_2(z)|^{q_2}\,d\mu_2(z)\le C\|f_2\|_{p_2,\a_2}^{q_2},
$$
or
$$
\int_{\Bn}|f_1(z)|^{q_1}|f_2(z)|^{q_2}\prod_{i=3}^k|h_i(z)|\,d\mu(z)
\le C\|f_1\|_{p_1,\a_1}^{q_1}\|f_2\|_{p_2,\a_2}^{q_2}\prod_{i=3}^k\|h_i\|_{p_i/q_i,\a_i}.
$$
Continuing this process we will eventually get (\ref{eq1}).
\end{proof}

\subsection{Proof of Theorem \ref{toeplitz1}}
\subsubsection{\textbf{(i) implies (ii)}}
We divide this part into two cases: $\lambda\ge1$ and $0<\lambda<1$.\\

\textbf{Case 1: $\lambda\ge 1$.}
Fix $a\in \Bn$ and let $f_a(z)=(1-\langle z,a \rangle )^{-(n+1+\beta)}$.
Under the condition $(n+1+\beta)p_1>n+1+\alpha_1$,
it is easy to check using Lemma \ref{Ict} that $f_a\in A^{p_1}_{\alpha_1}$ with
$$
\|f_a\|_{p_1,\alpha_1}^{p_1}\lesssim (1-|a|^2)^{(n+1+\alpha_1)-(n+1+\beta)p_1}.
$$
Since
\begin{displaymath}
\begin{split}
T_{\mu}^{\b}f_a(z)&=\int_{\Bn}\! \frac{f_a(w)}{(1-\langle z, w\rangle )^{n+1+\beta}}\,d\mu(w)
\\
&=\int_{\Bn} \!\frac{d\mu(w)}{(1-\langle z, w \rangle )^{n+1+\beta} \cdot (1-\langle w, a \rangle)^{n+1+\beta}},
\end{split}
\end{displaymath}
we get
$$
T_{\mu}^{\beta}f_z(z)=\int_{\Bn}\frac{d\mu(w)}{|1- \langle z,w \rangle |^{2(n+1+\beta)}}
\ge C\frac{\mu(D(z,r))}{(1-|z|^2)^{2(n+1+\beta)}}.
$$
On the other hand, by the pointwise estimate for functions in Bergman spaces
(see \cite[Theorem 2.1]{ZhuBn}) together with the boundedness of the Toeplitz operator $T_{\mu}^{\b}$, we get
\begin{eqnarray*}
T_{\mu}^{\beta}f_z(z)
&=&|T_{\mu}^{\beta}f_z(z)|
\le \|T_{\mu}^{\beta}f_z\|_{p_2,\alpha_2}(1-|z|^2)^{-(n+1+\alpha_2)/p_2}
\\
&\le&\|T_{\mu}^{\beta}\|\cdot \|f_z\|_{p_1,\alpha_1}(1-|z|^2)^{-(n+1+\alpha_2)/p_2}
\\
&\lesssim& \|T_{\mu}^{\beta}\|\,(1-|z|^2)^{(n+1+\alpha_1)/p_1-(n+1+\alpha_2)/p_2-(n+1+\beta)}.
\end{eqnarray*}
Hence
\begin{eqnarray*}
\mu(D(z,r))
&\lesssim&\|T_{\mu}^{\beta}\|\,(1-|z|^2)^{(n+1+\beta)+(n+1+\alpha_1)/p_1-(n+1+\alpha_2)/p_2}\\
&=&\|T_{\mu}^{\beta}\|\,(1-|z|^2)^{(n+1+\g)\lambda}.
\end{eqnarray*}
By Theorem A, this means that  $\mu$ is a $(\l,\g)$-Bergman Carleson measure with
$$
\|\mu\|_{\l,\gamma} \lesssim \|T_{\mu}^{\beta}\|.
$$

\textbf{Case 2: $0<\lambda<1$.}
Notice that the condition $0<\lambda<1$ is equivalent to $0<p_2<p_1<\infty$.
Let $r_k(t)$ be a sequence of Rademacher functions
(see \cite[Appendix A]{Du}) and $\{a_ k\}$ be any  $r$-lattice on $\Bn$.
Since $$n+1+\beta>n \max(1,1/p_ 1)+(1+\alpha_1)/p_1,$$ we know from Theorem 2.30 in \cite{ZhuBn} that,
for any sequence of real numbers $\{\lambda_k\}\in \ell^{p_1}$, the function
$$
f_t(z)=\sum_{k=1}^{\infty}\lambda_k r_k(t)
\frac{(1-|a_k|^2)^{n+1+\beta-(n+1+\alpha_1)/p_1}}{(1-\langle z, a_k \rangle)^{n+1+\beta}}
$$
is in $A^{p_1}_{\alpha_1}$ with $\|f_t\|_{p_1,\alpha_1}\lesssim\big \|\{\lambda_k\} \big \|_{\ell^{p_1}}$ for almost every $t$ in $(0,1)$.
Denote by
$$
f_k(z)=\frac{(1-|a_k|^2)^{n+1+\beta-(n+1+\alpha_1)/p_1}}{(1-\langle z, a_k \rangle )^{n+1+\beta}}.
$$
Since $T_{\mu}^{\b}$ is bounded from $A^{p_1}_{\a_1}$ to $A^{p_2}_{\a_2}$ then, for almost every $t$ in $(0,1)$,
we get that
\begin{eqnarray*}
\|T_{\mu}^{\beta}f_t\|_{p_2,\alpha_2}^{p_2}
&=&\int_{\Bn}\left|\sum_{k=1}^{\infty}\lambda_k \,r_k(t)\,T_{\mu}^{\beta}f_k(z)\right|^{p_2}\,dv_{\alpha_2}(z)
\\
&\lesssim & \|T_{\mu}^{\beta}\|^{p_ 2} \cdot \|f_t\|_{p_1,\alpha_1}^{p_2}
\lesssim \|T_{\mu}^{\beta}\|^{p_ 2} \left(\sum_{k=1}^{\infty}|\lambda_k|^{p_1}\right)^{p_2/p_1}.
\end{eqnarray*}
Integrating both sides with respect to $t$ from $0$ to $1$,
and using Fubini's Theorem and Khinchine's inequality (see \cite[Theorem A.2]{Du}),
we get
\begin{equation}\label{lp}
\int_{\Bn}\left(\sum_{k=1}^{\infty}|\lambda_k|^2
|T_{\mu}^{\beta}f_k(z)|^2\right)^{p_2/2}\!dv_{\alpha_2}(z)
\lesssim \|T_{\mu}^{\beta}\|^{p_ 2} \cdot \big \|\{\lambda_k\} \big \|_{\ell^{p_1}}^{p_2}.
\end{equation}
Let $\{D_k\}$
be the associated sets to the lattice $\{a_ k\}$ in Lemma~\ref{covering}.
Then
\begin{eqnarray}\label{t-mu0}
&~&\sum_{k=1}^{\infty}|\lambda_k|^{p_2}
\int_{\tilde D_k}|T_{\mu}^{\beta}f_k(z)|^{p_2}\,dv_{\alpha_2}(z)
\\
&~&\qquad =\int_{\Bn}\left(\sum_{k=1}^{\infty}|\lambda_k|^{p_2}
|T_{\mu}^{\beta}f_k(z)|^{p_2}\chi_{\tilde D_k}(z)\right)^{\frac{2}{p_2}\cdot\frac{p_2}{2}}
\!\!\!dv_{\alpha_2}(z). \notag
\end{eqnarray}
If $p_2\le 1$, then $2/p_2\le 1$, and from (\ref{t-mu0}) we have
\begin{eqnarray*}
&~&\sum_{k=1}^{\infty}|\lambda_k|^{p_2}
\int_{\tilde D_k}|T_{\mu}^{\beta}f_k(z)|^{p_2}\,dv_{\alpha_2}(z)
\\
&~&\qquad \le \int_{\Bn}\left(\sum_{k=1}^{\infty}|\lambda_k|^{2}
|T_{\mu}^{\beta}f_k(z)|^{2}\chi_{\tilde D_k}(z)\right)^{p_2/2}
\!\!dv_{\alpha_2}(z)
\\
&~&\qquad \le \int_{\Bn}\left(\sum_{k=1}^{\infty}|\lambda_k|^{2}
|T_{\mu}^{\beta}f_k(z)|^{2}\right)^{p_2/2}
\!\! dv_{\alpha_2}(z).
\end{eqnarray*}
If $0<p_2<1$, then $2/p_2>1$. Thus, from (\ref{t-mu0}) and  H\"older's inequality we get
\begin{eqnarray*}
&~&\sum_{k=1}^{\infty}|\lambda_k|^{p_2}
\int_{\tilde D_k}|T_{\mu}^{\beta}f_k(z)|^{p_2}\,dv_{\alpha_2}(z)
\\
&~&\qquad \le
\int_{\Bn}\left(\sum_{k=1}^{\infty}|\lambda_k|^2
|T_{\mu}^{\beta}f_k(z)|^2 \chi_{\tilde D_k(z)}\right)^{p_2/2}
\left(\sum_ k \chi_{\tilde D_k}(z)\right)^{1-p_2/2}
\!\!\!\!\!dv_{\alpha_2}(z)
\\
&~&\qquad \le N^{1-p_2/2}
\int_{\Bn}\left(\sum_{k=1}^{\infty}|\lambda_k|^2
|T_{\mu}^{\beta}f_k(z)|^2\right)^{p_2/2}\!\!\!dv_{\alpha_2}(z),
\end{eqnarray*}
since any point $z$ belongs to at most $N$ of the sets $\tilde{D}_ k$.
Combining the above two inequalities, and applying (\ref{lp}) we obtain
%The above inequality implies (since the covering has finite multiplicity)
\begin{eqnarray*}
&~&\sum_{k=1}^{\infty}|\lambda_k|^{p_2}
\int_{\tilde D_k}|T_{\mu}^{\beta}f_k(z)|^{p_2}\,dv_{\alpha_2}(z)
\\
&~&\qquad\le\min\{1,N^{1-p_2/2}\}
\int_{\Bn}\left(\sum_{k=1}^{\infty}|\lambda_k|^2
|T_{\mu}^{\beta}f_k(z)|^2\right)^{p_2/2}\!\!\!dv_{\alpha_2}(z)
\\
&~&\qquad\lesssim \|T_{\mu}^{\beta}\|^{p_ 2}\cdot \big \|\{\lambda_k\}\big \|_{\ell^{p_1}}^{p_2}.
\end{eqnarray*}
Since, by subharmonicity (see \cite[Lemma 2.24]{ZhuBn}), we have
$$
|T_{\mu}^{\beta}f_k(a_k)|^{p_2}
\lesssim \frac{1}{(1-|a_k|^2)^{n+1+\alpha_2}}
\int_{\tilde D_k}|T_{\mu}^{\beta}f_k(z)|^{p_2}\,dv_{\alpha_2}(z),
$$
we get
\begin{equation}\label{t-mu1}
\sum_{k=1}^{\infty}|\lambda_k|^{p_2}(1-|a_k|^2)^{n+1+\alpha_2}
|T_{\mu}^{\beta}f_k(a_k)|^{p_2}
\lesssim  \|T_{\mu}^{\beta}\|^{p_ 2}\cdot \big \|\{\lambda_k\} \big \|_{\ell^{p_1}}^{p_2}.
\end{equation}
Now, notice that
$$
T_{\mu}^{\beta}f_k(a_k)=
(1-|a_k|^2)^{n+1+\beta-(n+1+\alpha_1)/p_1}\int_{\Bn}\frac{d\mu(w)}{|1-\langle w,a_ k \rangle |^{2(n+1+\beta)}}.
$$
Therefore
$$
\frac{\mu(D_k)}{(1-|a_k|^2)^{n+1+\beta+(n+1+\alpha_1)/p_1}}\lesssim \,T_{\mu}^{\beta}f_k(a_k),
$$
and putting this into (\ref{t-mu1}) we get
$$
\sum_{k=1}^{\infty}|\lambda_k|^{p_2}\left(\frac{\mu(D_k)}{(1-|a_k|^2)^s}\right)^{p_2}
\lesssim \|T_{\mu}^{\beta}\|^{p_ 2}\cdot \big \|\{\lambda_k\}\big \|_{\ell^{p_1}}^{p_2},
$$
with
\begin{equation}\label{EqS}
s=n+1+\beta+\frac{(n+1+\alpha_1)}{p_1}-\frac{(n+1+\alpha_2)}{p_2}=(n+1+\gamma)\lambda.
\end{equation}
Since the conjugate exponent of $(p_ 1/p_ 2)$ is $(p_1/p_2)'=p_1/(p_1-p_2)$,
by duality, we know that
$$
\{\nu_ k\}:=\left\{\left(\frac{\mu(D_k)}{(1-|a_k|^2)^s}\right)^{p_2}\right\}\in \ell^{p_1/(p_1-p_2)}
$$
with
$$
\big \|\{\nu_ k\} \big \|_{\ell^{p_1/(p_1-p_2)}} \lesssim \|T_{\mu}^{\beta}\|^{p_ 2},
$$
or,
$$
\{\mu_ k\}:=\left\{\frac{\mu(D_k)}{(1-|a_k|^2)^{(n+1+\gamma)\lambda}}\right\}\in \ell^{p_1p_2/(p_1-p_2)}=\ell^{1/(1-\lambda)}
$$
with
$$\big \|\{\mu_ k\} \big \|_{\ell^{1/(1-\l)}}=\big \|\{\nu_ k\} \big \|_{\ell^{p_1/(p_1-p_2)}}^{1/p_ 2} \lesssim \|T_{\mu}^{\beta}\|.$$
By Theorem B, this means that
$\mu$ is an $(\lambda,\gamma)$-Bergman Carleson measure with
$$\|\mu\|_{\l,\gamma} \lesssim \|T_{\mu}^{\beta}\|.$$

\subsubsection{\textbf{(ii) implies (i)}}
Now suppose (ii) holds, i.e., $\mu$ is a $(\l,\g)$-Bergman Carleson measure, we prove (i).
We divide the proof into three cases.\\

\textbf{Case 1: $p_2>1$.}
For this case, let $p'_2$ and $\a'_2$ be two numbers satisfying
\begin{equation}\label{Eq-DN}
\frac{1}{p_2}+\frac{1}{p'_2}=1;\qquad
\frac{\a_2}{p_2}+\frac{\a'_2}{p'_2}=\b.
\end{equation}
Then
$$
\a'_2=\left(\b-\frac{\a_2}{p_2}\right)p'_2=\frac{\b p_2-\a_2}{p_2-1}>-1
$$
since $\b>(1+\a_2)/p_2-1$. By a duality result due to Luecking (see \cite{l1-b} or \cite[Theorem 2.12]{ZhuBn}), we know that
$(A^{p_2}_{\a_2})^*=A^{p'_2}_{\a'_2}$ under the integral pairing
$$
\langle f, g\rangle_{\b}=\int_{\Bn} f(z)\,\overline{g(z)}\,dv_{\b}(z).
$$

Let $f\in A^{p_1}_{\a_1}$ and $h\in A^{p'_2}_{\a'_2}$.
An easy computation using Fubini's Theorem and the reproducing formula for Bergman spaces shows
$$
\langle h, T_{\mu}^{\beta}f\rangle_{\b}=\int_{\Bn} h(z)\,\overline{f(z)}\,d\mu(z).
$$
The conditions for $\l$ and $\g$ in the Theorem are equivalent to
$$
\l=\frac1{p_1}+\frac1{p'_2}, \qquad
\g=\frac1{\l}\left(\frac{\a_1}{p_1}+\frac{\a'_2}{p'_2}\right).
$$
Thus by Proposition \ref{Car-P1},
\begin{displaymath}
\begin{split}
|\langle h, T_{\mu}^{\beta}f\rangle_{\b}|
&\le \int_{\Bn} |h(z)|\,|f(z)|\,d\mu(z)
\lesssim \|\mu\|_{\l,\g}\cdot \|f\|_{p_1,\a_1}\cdot \|h\|_{p'_2,\a'_2}.
\end{split}
\end{displaymath}
Hence $T_{\mu}^{\beta}$ is bounded from $A^{p_1}_{\a_1}$ to $A^{p_2}_{\a_2}$ with $\|T_{\mu}^{\beta}\|\lesssim \|\mu\|_{\l,\g}$.\\

\textbf{Case 2: $p_2=1$.}
Let $f\in A^{p_1}_{\alpha_1}$.
For this case,
since $\beta>(1+\alpha_2)/1-1=\alpha_2$,
by Fubini's Theorem and Lemma \ref{Ict} we have
\begin{eqnarray}\label{eq8}
\|T_{\mu}^{\beta}f\|_{1,\alpha_2}
&\le&\int_{\Bn}\left(\int_{\Bn}\frac{|f(w)|}{|1-\langle z, w \rangle |^{n+1+\beta}}\,d\mu(w)\right)\,dv_{\alpha_2}(z)
\\
&=&\int_{\Bn}|f(w)|\left(\int_{\Bn}\frac{(1-|z|^2)^{\alpha_2}}{|1-\langle z, w\rangle|^{n+1+\beta}}\,dv(z)\right)\,d\mu(w) \notag
\\
&\lesssim&\int_{\Bn}|f(w)|(1-|w|^2)^{\alpha_2-\beta}\,d\mu(w). \notag
\end{eqnarray}
Let $\nu$ be the measure defined by $d \nu(w)=(1-|w|^2)^{\alpha_2-\beta}\,d\mu(w)$.
Since $\mu$ is a $(\lambda,\gamma)$-Bergman Carleson measure, using Theorems A and B it is easy to see that $\nu$ is a $(1/p_1,\alpha_1)$-Bergman Carleson measure, and moreover, $\|\nu\|_{1/p_ 1,\a _ 1}\lesssim \|\mu\|_{\l,\g}$. Thus, for any $f\in A^{p_1}_{\alpha_1}$, we have
$$
\int_{\Bn}|f(w)|\,d \nu(w)
\lesssim \|\nu\|_{1/p_ 1,\a _ 1}\cdot \|f\|_{p_1,\alpha_1}\lesssim \|\mu\|_{\l,\g}\cdot \|f\|_{p_1,\alpha_1}.
$$
Thus, by (\ref{eq8}), it follows that
$$
\|T_{\mu}^{\beta}f\|_{1,\alpha_2}\lesssim \|\mu\|_{\l,\g}\cdot \|f\|_{p_1,\alpha_1},
$$
and so $T_{\mu}^{\beta}$ is bounded from $A^{p_1}_{\a_1}$ to $A^{p_2}_{\a_2}$ with $\|T_{\mu}^{\beta}\|\lesssim \|\mu\|_{\l,\g}$.\\

\textbf{Case 3: $0<p_2<1$.}
Let $\{a_k\}$ be an $r$-lattice of $\Bn$ in the Bergman metric
and $\{D_k\}$ be the corresponding sets as in Lemma~\ref{covering}.
Then we know that $\Bn=\cup_{k=1}^{\infty}D_k$
and there is a positive integer $N$ such that
each point in $\Bn$ belongs to at most $N$ of
sets $\tilde D_k$.
Then
\begin{displaymath}
\begin{split}
|T_{\mu}^{\beta}f(z)| &
\lesssim\sum_{k=1}^{\infty}\int_{D_k}\frac{|f(w)|}{|1-\langle z,w\rangle |^{n+1+\beta}}\,d\mu(w)
\\
&
\lesssim \sum_{k=1}^{\infty}\frac1{|1-\langle z, a_k \rangle|^{n+1+\beta}}\int_{D_k}|f(w)|\,d\mu(w).
\end{split}
\end{displaymath}
Now, for $w\in D_k$, one has
$$
|f(w)|^{p_1}
\lesssim \frac1{(1-|a_k|^2)^{n+1+\alpha_1}}\int_{\tilde{D}_k}|f(z)|^{p_1}\,dv_{\alpha_1}(z).
$$
From this we get
$$
\int_{D_k}|f(w)|\,d\mu(w)
\lesssim  \frac1{(1-|a_k|^2)^{(n+1+\alpha_1)/p_1}}
\left(\int_{\tilde{D}_k}|f(z)|^{p_1}\,dv_{\alpha_1}(z)\right)^{1/p_1}\mu(D_k).
$$
Since $0<p_2<1$, this implies
$$
|T_{\mu}^{\beta}f(z)|^{p_2}
\lesssim \sum_{k=1}^{\infty}\frac1{|1-\langle z, a_k\rangle |^{(n+1+\beta)p_2}}
\frac{\mu(D_k)^{p_2}}{(1-|a_k|^2)^{(n+1+\alpha_1)\frac{p_2}{p_1}}}
\left(\int_{\tilde{D}_k} \!\!|f(z)|^{p_1}\,dv_{\alpha_1}(z)\right)^{\frac{p_2}{p_1}}.
$$
Therefore, since $(n+1+\beta)p_2>n+1+\alpha_2$, we can apply Lemma \ref{Ict} to obtain
\begin{eqnarray}\label{t-mu2}
\|T_{\mu}^{\beta}f\|_{p_2,\alpha_2}^{p_2}
&\lesssim &\sum_{k=1}^{\infty}
\frac{\mu(D_k)^{p_2}}{(1-|a_k|^2)^{(n+1+\alpha_1)p_2/p_1}}
\left(\int_{\tilde D_k}  \! \! |f(z)|^{p_1}\,dv_{\alpha_1}(z)\right)^{p_2/p_1} \notag
\\
&~&\quad\times
\int_{\Bn}\frac{(1-|z|^2)^{\alpha_2}\,dv(z)}{|1-\langle z, a_k \rangle |^{(n+1+\beta)p_2}}
\\
&\lesssim &\sum_{k=1}^{\infty}
\frac{\mu(D_k)^{p_2}}{(1-|a_k|^2)^{(n+1+\alpha_1)p_2/p_1}}
\left(\int_{\tilde D_k}  \!\!  |f(z)|^{p_1}\,dv_{\alpha_1}(z)\right)^{p_2/p_1} \notag
\\
&~&\quad\times
(1-|a_k|^2)^{n+1+\alpha_2-(n+1+\beta)p_2}. \notag
\end{eqnarray}
First assume that $\lambda\ge 1$.
Since $\mu$ is a $(\lambda,\gamma)$-Bergman Carleson measure,
by Theorem A we get
$$
\mu(D_ k)\lesssim \|\mu\|_{\l,\g}\,(1-|a_ k|^2)^{(n+1+\gamma)\l}.
$$
Bearing in mind \eqref{EqS}, this together with \eqref{t-mu2} and the fact that $p_ 2\ge p_ 1$ (due to the assumption $\l \ge 1$), yields
\begin{eqnarray*}
\|T_{\mu}^{\beta}f\|_{p_2,\alpha_2}^{p_2}
&\lesssim & \|\mu\|_{\l,\g}^{p_ 2}\,\sum_{k=1}^{\infty}
\left(\int_{\tilde D_k}|f(z)|^{p_1}\,dv_{\alpha_1}(z)\right)^{p_2/p_1}
\\
&\lesssim & \|\mu\|_{\l,\g}^{p_ 2}\,\left(\sum_{k=1}^{\infty}
\int_{\tilde D_k}|f(z)|^{p_1}\,dv_{\alpha_1}(z)\right)^{p_2/p_1} \lesssim  \|\mu\|_{\l,\g}^{p_ 2}\cdot\|f\|_{p_1,\alpha_1}^{p_2}.
\end{eqnarray*}
Hence $T_{\mu}^{\beta}$ is bounded from $A^{p_1}_{\alpha_1}$ to $A^{p_2}_{\alpha_2}$ with $\|T_{\mu}^{\beta}\|\lesssim \|\mu\|_{\l,\g}$.
\\
\\
Next assume that $0<\lambda<1$. Then $p_ 1>p_ 2$,
and using H\"older's inequality in \eqref{t-mu2} we get
\begin{eqnarray*}
\|T_{\mu}^{\beta}f\|_{p_2,\alpha_2}^{p_2}
&\lesssim &\sum_{k=1}^{\infty}
\frac{\mu(D_k)^{p_2}}{(1-|a_k|^2)^{(n+1+\gamma)\l p_2}}
\left(\int_{\tilde D_k}|f(z)|^{p_1}\,dv_{\alpha_1}(z)\right)^{p_2/p_1}
\\
&\le&\left\{\sum_{k=1}^{\infty}
\left[\frac{\mu(D_k)^{p_2}}{(1-|a_k|^2)^{(n+1+\gamma)\l p_2}}\right]^{p_1/(p_1-p_2)}
\right\}^{1-p_2/p_1}
\\
&~&\quad\times \left (\sum_{k=1}^{\infty}
\int_{\tilde D_k}|f(z)|^{p_1}\,dv_{\alpha_1}(z)\right)^{p_ 2/p_ 1}.
\end{eqnarray*}
Since $\mu$ is a $(\lambda,\gamma)$-Bergman Carleson measure,
by Theorem B we get that
\begin{displaymath}
\begin{split}
\sum_{k=1}^{\infty}
\left[\frac{\mu(D_k)^{p_2}}{(1-|a_k|^2)^{(n+1+\gamma)\l p_2}}\right]^{p_1/(p_1-p_2)}
&=\sum_{k=1}^{\infty}
\left[\frac{\mu(D_k)}{(1-|a_k|^2)^{(n+1+\gamma)\lambda}}\right]^{1/(1-\lambda)}\\
\\
&\lesssim  \|\mu\|_{\l,\g}^{1/(1-\lambda)}=\|\mu\|_{\l,\g}^{p_ 1 p_ 2/(p_ 1-p_ 2)},
\end{split}
\end{displaymath}
and so
$$
\|T_{\mu}^{\beta}f\|_{p_2,\alpha_2}^{p_2}
\lesssim  \|\mu\|_{\l,\g}^{p_ 2}\,\left (\sum_{k=1}^{\infty}
\int_{\tilde D_k}|f(z)|^{p_1}\,dv_{\alpha_1}(z)\right)^{p_2/p_1}
\lesssim  \|\mu\|_{\l,\g}^{p_ 2}\cdot \|f\|_{p_1,\alpha_1}^{p_2}.
$$
Hence $T_{\mu}^{\beta}$ is bounded from $A^{p_1}_{\alpha_1}$ to $A^{p_2}_{\alpha_2}$ with $\|T_{\mu}^{\beta}\|\lesssim \|\mu\|_{\l,\g}$.
The proof is complete.

\subsection{A key lemma}
Now, we are going to use the result just proved on Toeplitz operators to obtain the following technical result that will be the key for the proof of the remaining part in Theorem \ref{carleson}.
\begin{lemma}\label{ML}
Let $\mu$ be a positive Borel measure on the
unit ball $\Bn$. For $s,r>0$ and $\alpha _ 1>-1$, let
$$
S_{\mu,\alpha _ 1}^r f(z)
=(1-|z|^2)^s\int_{\Bn} \frac{|f(w)|^r\,d\mu(w)}{|1-\langle z,w \rangle |^{n+1+s+\alpha_ 1}}.
$$
For $q>1$, $p>0$ and $\a_ 2>-1$, let
\begin{equation}\label{gl}
\lambda=1+\frac{r}{p}-\frac{1}{q},\quad \textrm{and}\quad \gamma
=\frac{1}{\lambda}\Big(\alpha_ 1+\frac{\alpha_ 2 r}{p}-\frac{\alpha_ 1}{q} \Big ).
\end{equation}
The following conditions are equivalent:
\begin{enumerate}
\item [(a)] $\mu$ is a $(\lambda,\gamma)$-Bergman Carleson measure.
\item[(b)] There is a positive constant $K$ such that
$\|S_{\mu,\a _ 1}^r f \|_{q,\a _ 1}\le K \|f\|_{p,\a_ 2}^r$ for $f\in A^p_{\a_ 2}$.
\end{enumerate}
Moreover, one has $\|\mu\|_{\l,\g} \asymp  K$.
\end{lemma}

\begin{proof}
Suppose first that $\mu$ is a $(\lambda,\gamma)$-Bergman Carleson measure. Consider a lattice $\{a_ j\}$ and its associated sets $\{D_ j\}$. Since $|1-\langle z,w \rangle |$ is comparable with $|1-\langle z,a_ j\rangle  |$ for $w$ in $D_ j$, we have
\begin{displaymath}
\begin{split}
|S_{\mu,\alpha_ 1}^r f(z)|
&\lesssim (1-|z|^2)^s \sum_{j=1}^{\infty} \int_{D_j} \frac{|f(w)|^r\,d\mu(w)}{|1-\langle z,w \rangle |^{n+1+s+\alpha_ 1}}
\\
&\asymp (1-|z|^2)^s \sum_{j=1}^{\infty}\frac{1}{|1-\langle z,a_j \rangle |^{n+1+s+\alpha_ 1}}\int_{D_ j} |f(w)|^r\,d\mu(w).
\end{split}
\end{displaymath}
Using the notation
$$
|\widehat{f}(a_ j)|
:=\left (\frac{1}{(1-|a_j|^2)^{n+1+\alpha_ 2}}\int_{\tilde{D}_ j} |f(\zeta)|^p\,dv_{\alpha_ 2}(\zeta) \right )^{1/p}
$$
we have
\begin{displaymath}
|f(w)|^r \lesssim |\widehat{f}(a_ j)|^r \quad w\in D_ j.
\end{displaymath}
This gives
\begin{displaymath}
|S_{\mu,\alpha_ 1}^r f(z)|^q
\lesssim (1-|z|^2)^{sq} \left ( \sum_{j=1}^{\infty} \frac{|\widehat{f}(a_ j)|^r\, \mu(D_ j)}{|1-\langle z,a_ j \rangle|^{n+1+s+\alpha_ 1}}\right )^q.
\end{displaymath}
Now, pick $\varepsilon>0$ so that $\alpha_ 1-\varepsilon \max(q,q')>-1$,
with $q'$ being the conjugate exponent of $q$, that is, $1/q+1/q'=1$. By H\"{o}lder's inequality with exponent $q>1$ we get
\begin{displaymath}
\begin{split}
\left ( \sum_{j=1}^{\infty}\frac{|\widehat{f}(a_ j)|^r\, \mu(D_ j)}{|1-\langle z,a_ j \rangle|^{n+1+s+\alpha_ 1}}\right )^q
&
\le \left ( \sum_{j=1}^{\infty} \frac{(1-|a_ j|^2)^{n+1+\alpha_ 1-\varepsilon q'}}{|1-\langle z, a_ j \rangle|^{n+1+s+\alpha_ 1}}\right )^{q-1}
\\
&
\,\,\,\times \left ( \sum_{j=1}^{\infty} \frac{|\widehat{f}(a_ j)|^{rq}\, \mu(D_ j)^q\,(1-|a_ j|^2)^{(n+1+\alpha_ 1)(1-q)+\varepsilon q}}
{|1-\langle z,a_ j \rangle |^{n+1+s+\alpha_ 1}}\right ).
\end{split}
\end{displaymath}
Since the sequence $\{a_ j\}$ is separated and $n+1+\alpha_ 1-\varepsilon q'>n$, using Lemma~\ref{l2} we have
$$
\sum_{j=1}^{\infty} \frac{(1-|a_ j|^2)^{n+1+\alpha_ 1-\varepsilon q'}}{|1-\langle z,a_ j \rangle |^{n+1+s+\alpha_ 1}}
\lesssim (1-|z|^2)^{-s-\varepsilon q'},
$$
and therefore
\begin{displaymath}
|S_{\mu,\alpha_ 1}^r f(z)|^q \lesssim (1-|z|^2)^{s-\varepsilon q}
\left( \sum_{j=1}^{\infty} \frac{|\widehat{f}(a_ j)|^{rq}\, \mu(D_ j)^q\,(1-|a_ j|^2)^{(n+1+\alpha_ 1)(1-q)+\varepsilon q}}
{|1-\langle z,a_j \rangle|^{n+1+s+\alpha_ 1}}\right).
\end{displaymath}
This together with the typical integral estimate in Lemma \ref{Ict} gives
\begin{displaymath}
\begin{split}
\|S_{\mu,\alpha_ 1}^r f\|_{q,\alpha_ 1}^q
&\lesssim \sum_{j=1}^{\infty} |\widehat{f}(a_ j)|^{rq}\, \mu(D_ j)^q\,(1-|a_ j|^2)^{(n+1+\alpha_ 1)(1-q)+\varepsilon q}
\\
& \,\, \,\,\times
\int_{\Bn} \frac{(1-|z|^2)^{s+\alpha_ 1-\varepsilon q}}{|1-\langle z,a_ j \rangle |^{n+1+s+\alpha_ 1}} dv(z)
\\
& \lesssim \sum_{j=1}^{\infty} |\widehat{f}(a_ j)|^{rq}\, \mu(D_ j)^q\,(1-|a_ j|^2)^{(n+1+\alpha_ 1)(1-q)} .
\end{split}
\end{displaymath}
\par If $\l \ge 1$, then $\mu(D_ j) \lesssim \|\mu\|_{\l,\g} (1-|a_ j|^2)^{(n+1+\g)\l}$
due to Theorem \ref{TA}. Moreover, the condition $\l \ge 1$ also implies $p/rq\le 1$, and therefore we have
\begin{displaymath}
\begin{split}
\|S_{\mu,\alpha_ 1}^r f\|_{q,\alpha_ 1}^q& \lesssim \|\mu\|_{\l,\g}^q \,
\sum_{j=1}^{\infty} |\widehat{f}(a_ j)|^{rq}\,(1-|a_ j|^2)^{(n+1+\alpha_ 2)rq/p}
\\
& \le \|\mu\|_{\l,\g}^q \,\left (\sum_{j=1}^{\infty} |\widehat{f}(a_j)|^{p}(1-|a_ j|^2)^{n+1+\alpha_ 2}\right )^{\frac{rq}{p}}.
\end{split}
\end{displaymath}

If $0<\l <1$ we use H\"{o}lder's inequality with exponent $p/rq>1$. Observe that the conjugate exponent of $p/(rq)$ is
$$
\frac{p/(rq)}{p/(rq)-1}=\frac{p}{p-rq}=\frac{1}{q(1-\lambda)}.
$$
We obtain, after an application of Theorem \ref{TB},
\begin{displaymath}
\begin{split}
\|S_{\mu,\alpha_ 1}^r f\|_{q,\alpha_ 1}^q
&\lesssim
\left (\sum_{j=1}^{\infty} |\widehat{f}(a_ j)|^{p}(1-|a_ j|^2)^{n+1+\alpha_ 2}\right )^{\frac{rq}{p}}
\left (\sum_{j=1}^{\infty} \left (\frac{\mu(D_ j)}{(1-|a_ j|^2)^{(n+1+\g)\l}} \right )^{\frac{1}{1-\l}}\right )^{q(1-\l)}
\\
& \lesssim \|\mu\|_{\l,\g}^q \,\left (\sum_{j=1}^{\infty}|\widehat{f}(a_ j)|^{p}(1-|a_ j|^2)^{n+1+\alpha_ 2}\right )^{\frac{rq}{p}}.
\end{split}
\end{displaymath}
Finally, in both cases, we obtain the inequality in part (b) after noticing that
$$\sum_{j=1}^{\infty} |\widehat{f}(a_ j)|^{p}(1-|a_ j|^2)^{n+1+\alpha_ 2}\lesssim \|f\|^p_{p,\a_ 2}.$$
\par Conversely, assume that (b) holds. We want to show that $\mu$ is a $(\l,\g)$-Bergman Carleson measure. We split the proof in two cases.

If $\l\ge 1$, for each $a\in \Bn$, consider the functions
$$f_ a(z)=(1-\langle z,a \rangle )^{-\sigma}$$
with $\sigma$ big enough, that is, with $p\sigma>n+1+\a_ 2$. By Lemma \ref{Ict} we have
$$ \|f_ a\|_{p,\a_ 2}^p =\int_{\Bn} \frac{dv_{\a _ 2}(z)}{|1-\langle z,a \rangle |^{p\sigma}} \lesssim (1-|a|^2)^{n+1+\a_ 2 -p\sigma}.$$
Also, for any $\tau>0$ we get
\begin{displaymath}
\begin{split}
\frac{(1-|z|^2)^s |f_ a(a)|^r}{|1-\langle z,a\rangle |^{n+1+s+\a _ 1}} \mu(D(a,\tau)) &\lesssim (1-|z|^2)^s\int_{D(a,\tau)} \frac{|f_ a(w)|^r\,d\mu(w)}{|1-\langle z,w\rangle |^{n+1+s+\a _ 1}}
\\
& \le S_{\mu,\a _ 1}^r f_ a(z).
\end{split}
\end{displaymath}
Moreover, since $v_{\a _ 1}(D(a,\tau))\asymp (1-|a|^2)^{n+1+\a_ 1}$, we have
\begin{displaymath}
\begin{split}
(1-|a|^2)^{(n+1+\a_ 1)(1-q)} & \lesssim \int_{D(a,\tau)} \left (\frac{(1-|z|^2)^s }{|1-\langle z,a\rangle |^{n+1+s+\a _ 1}}\right )^q\,dv_{\a _ 1}(z)
\\
&\le \int_{\Bn} \left (\frac{(1-|z|^2)^s }{|1-\langle z,a\rangle |^{n+1+s+\a _ 1}}\right )^q\,dv_{\a _ 1}(z).
\end{split}
\end{displaymath}
Hence,
\begin{displaymath}
\begin{split}
(1-|a|^2)^{(n+1+\a_ 1)(1-q)}\,&|f_ a(a)|^{rq}\,\mu(D(a,\tau))^q \lesssim \int_{\Bn}S_{\mu,\a _ 1}^r f_ a(z) ^q \,dv_{\a_ 1} (z)
\\
& \le K^q \,\|f_ a\|^{rq}_{p,\a _ 2} \lesssim K^q \Big ((1-|a|^2)^{n+1+\a_ 2 -p\sigma}\Big )^{rq/p}.
\end{split}
\end{displaymath}
This gives
\begin{displaymath}
\begin{split}
\mu(D(a,\tau)) &\lesssim K \, (1-|a|^2)^{(n+1+\a_ 1)\frac{(q-1)}{q}}\,(1-|a|^2)^{(n+1+\a_ 2)r/p}
\\
&=K\,(1-|a|^2)^{(n+1+\gamma)\lambda}.
\end{split}
\end{displaymath}
By Theorem \ref{TA} it follows that $\mu$ is a $(\l,\g)$-Bergman Carleson measure with $\|\mu\|_{\l,\g}\lesssim K$.\\

For  $0<\l<1$,
we split the proof in several cases.\\

\emph{Case $r=1$}: In that case, the condition easily implies that
the Toeplitz operator
$T_{\mu}^{\beta}:A^{p}_{\alpha_ 2}\rightarrow A^{q}_{\sigma}$ is bounded,
with
$$
\beta=s+\alpha_ 1 \quad  \textrm{and} \quad \sigma=\alpha_ 1 +sq.
$$
Therefore, part (a) is an immediate consequence of Theorem \ref{toeplitz1}. \\

\emph{Case $r>1$}: We want to show that $\mu$ is a $(\l,\g)$-Bergman Carleson measure,
or equivalently, that $B_{s,\gamma}(\mu)$ belongs to $L^{1/(1-\l),\g}$.
By Theorem \ref{TB} and the result on Toeplitz operators (Theorem \ref{toeplitz1}), we have
\begin{equation}\label{EqKL-2}
\|B_{s,\gamma}(\mu)\|_{\frac{1}{1-\lambda},\gamma}\le C \|T^{\beta}_{\mu}\|_{A^{p}_{\alpha_ 2}\rightarrow A^{t}_{\sigma}}
\end{equation}
with
\begin{displaymath}
\begin{array}{ccl}
\beta&=&\displaystyle{s+\frac{\alpha_ 1}{r}+\frac{\gamma (r-1)}{r} }\\
\\
t&=&\displaystyle{\frac{1}{1-\lambda+\frac{1}{p}}}
\end{array}
\end{displaymath}
and $\sigma$ is determined by the relation
$$
\gamma=\frac{1}{\lambda} \left (\beta+\frac{\alpha_ 2}{p}-\frac{\sigma}{t}\right ).
$$
Assume first that $\mu$ has compact support on $\Bn$.
By H\"{o}lder's inequality,
\begin{displaymath}
|T_{\mu}^{\beta} f(z)|^{t}
\le \left (\int_{\Bn} \frac{|f(w)|^{r}\, d\mu(w)}{|1-\langle z,w\rangle |^{n+1+s+\alpha_ 1}}\right )^{\frac{t}{r}}
\left (\int_{\Bn} \frac{d\mu(w)}{|1-\langle z,w \rangle |^{n+1+s+\gamma}} \right )^{\frac{t}{r'}},
\end{displaymath}
where $r'=\frac{r}{r-1}$ is the conjugate exponent of $r$. This yields
\begin{displaymath}
\|T_{\mu}^{\beta} f \|_{t,\sigma}^{t}
\le \int_{\Bn} \left (\int_{\Bn}\frac{(1-|z|^2)^s|f(w)|^{r} d\mu(w)}{|1-\langle z,w \rangle|^{n+1+s+\alpha_ 1}}\right )^{\frac{t}{r}}
\left (B_{s,\gamma} \mu (z) \right )^{\frac{t}{r'}} dv_{\sigma-st}(z).
\end{displaymath}
Now, since $\frac{r'}{(1-\l)t}>1$ (because $(1-\l)t<1$) we can apply H\"{o}lder's inequality again to obtain
\begin{displaymath}
\|T_{\mu}^{\beta} f \|_{t,\sigma}^{t} \le \big \|B_{s,\gamma}(\mu)\big \|_{\frac{1}{1-\lambda},\gamma}^{\frac{t}{r'}}\left [
\int_{\Bn} \!\!\left (\int_{\Bn} \!\frac{(1-|z|^2)^s|f(w)|^{r} d\mu(w)}{|1-\langle z,w \rangle |^{n+1+s+\alpha_ 1}}\right )^{q}
dv_{\gamma+(\sigma-st-\gamma)\eta}(z)\right ]^{1/\eta}.
\end{displaymath}
Observe that
$$\eta:=\left (\frac{r'}{(1-\l)\,t}\right )'=\frac{\frac{r'}{(1-\l)\,t}}{\frac{r'}{(1-\l)\,t}-1}=\frac{r'}{r'-(1-\l)\,t}
=\frac{r}{r-(1-\l)\,t\,(r-1)}$$
and therefore
\begin{displaymath}
\begin{split}
\frac{t \eta}{r} &=\frac{t}{r-(1-\l)\,t\,(r-1)}=\frac{1}{\frac{r}{t}-(1-\l)(r-1)}
\\
&=\frac{1}{r(1-\l +\frac{1}{p})-(1-\l)(r-1)}=\frac{1}{\frac{r}{p}+1-\l}=q.
\end{split}
\end{displaymath}
After some long and tedious but elementary computations it is possible to check that
\begin{equation}\label{D-Id}
\gamma+(\sigma-st-\gamma)\eta=\alpha_ 1.
\end{equation}
Indeed, since $1-\eta=\frac{-(1-\l)t(r-1)}{r-(1-\l)t(r-1)}$, the identity \eqref{D-Id} is equivalent to
$$ -\gamma (1-\l) t(r-1)+\sigma r -str=\a _ 1 [r-(1-\l)t(r-1)].$$
Using that $\sigma=\beta t +\a _ 2 t/p-\g \l t$ and the expression of $\beta$, after some simplifications we see that the previous identity is equivalent to
$$-\g \l t +\a_ 2 \,t\,\frac{r}{p}=\a_ 1 r -\a_ 1 tr (1-\l)-\a_ 1 \l t.$$
Now, using the expressions of $\l,\g$ given in \eqref{gl}, we must check that
$$\frac{\a _ 1 t}{p}=\a _ 1  -\a_ 1 t  (1-\l).$$
This is obvious if $\a_ 1=0$. If $\a_ 1\neq 0$ this is equivalent to
$$t=\frac{1}{1-\l +\frac{1}{p}},$$
and this is our choice of $t$. Hence \eqref{D-Id} holds.\\

Then, by our condition (b) we obtain
\begin{displaymath}
\begin{split}
\|T_{\mu}^{\beta} f\|_{t,\sigma}^{t} & \le \big \|B_{s,\gamma}(\mu)\big \|_{\frac{1}{1-\lambda},\gamma}^{\frac{t}{r'}}\cdot
  \big \|S_{\mu,\a_ 1}^rf \big \|_{q,\alpha_ 1}^{\frac{q }{\eta}}
\\
&\le K^{q/\eta}\cdot \big \|B_{s,\gamma}(\mu)\big \|_{\frac{1}{1-\lambda},\gamma}^{\frac{t}{r'}}\cdot  \|f\|_{p,\alpha_ 2}^{t}.
\end{split}
\end{displaymath}
This together with \eqref{EqKL-2} gives
$$\|B_{s,\gamma}(\mu)\|_{\frac{1}{1-\lambda},\gamma}\lesssim \|T_{\mu}^{\beta}\|\lesssim K^{q/\eta t}\cdot \big \|B_{s,\gamma}(\mu) \big \|_{\frac{1}{1-\lambda},\gamma}^{\frac{1}{r'}},$$
and since $q/\eta t=1/r$ this implies
$$\|B_{s,\gamma}(\mu)\|_{\frac{1}{1-\lambda},\gamma}\lesssim K$$
proving the result when $\mu$ has compact support on $\Bn$. The result for arbitrary $\mu$ follows from this by an easy limit argument.
\\

\emph{Case $r<1$}: Fix a number $m>1$ and consider the measure $\widetilde{\mu}$ given by
$$
d\widetilde{\mu}(z)=(1-|z|^2)^A\,d\mu(z)
$$
with
$$
A=(m-r)\frac{(n+1+\a_ 2)}{p}.
$$
Let
\begin{displaymath}
\begin{array}{ccl}
\gamma^*&=&\gamma +\frac{A}{\l}\\
\\
\beta&=&\displaystyle{s+\frac{\alpha_ 1}{m}+\frac{\gamma ^* (m-1)}{m} }\\
\\
t&=&\displaystyle{\frac{1}{1-\lambda+\frac{1}{p}}}
\end{array}
\end{displaymath}
and let $\sigma$  be determined by the relation
$$\gamma^*=\frac{1}{\lambda} \left (\beta+\frac{\alpha_ 2}{p}-\frac{\sigma}{t}\right ).$$
Again, assume first that $\mu$ has compact support on $\Bn$.
Obviously, then the measure $\widetilde{\mu}$ has also compact support.
By Theorem \ref{toeplitz1} applied to the Toeplitz operator $T^{\beta}_{\widetilde{\mu}}:A^p_{\a_ 2}\rightarrow A^t_{\sigma}$ we have
 \begin{equation}\label{EqKL-3}
\|B_{s,\gamma ^*}(\widetilde{\mu})\|_{\frac{1}{1-\lambda},\gamma ^*}\le C \|T^{\beta}_{\widetilde{\mu}}\|_{A^{p}_{\alpha_ 2}\rightarrow A^{t}_{\sigma}}.
\end{equation}
Arguing as in the previous case, we get
\begin{equation*}
\|T_{\widetilde{\mu}}^{\beta} f\|_{t,\sigma}^{t}  \le \big \|B_{s,\gamma ^*}(\widetilde{\mu})\big \|_{\frac{1}{1-\lambda},\gamma^*}^{\frac{t}{m'}}\cdot  \big \|S_{\widetilde{\mu},\a_ 1}^m f \big \|_{q_ 1,\alpha_ 1}^{\frac{t }{m}},
\end{equation*}
where $m'$ denotes the conjugate exponent of $m$ and
$$q_ 1=\frac{1}{\frac{m}{p}+1-\l}.$$
Since $m>r$ one has $q_ 1<(r/p+1-\l)^{-1}=q$, and hence
$$\|S_{\widetilde{\mu},\a_ 1}^m f\|_{q_ 1,\alpha_ 1}\le \|S_{\widetilde{\mu},\a_ 1}^m f\|_{q,\alpha_ 1}.$$
Therefore,
\begin{equation}\label{Eq-CII}
\|T_{\widetilde{\mu}}^{\beta} f\|_{t,\sigma}^{t}  \le \big \|B_{s,\gamma^*}(\widetilde{\mu})\big \|_{\frac{1}{1-\lambda},\gamma^*}^{\frac{t}{m'}}\cdot  \big \|S_{\widetilde{\mu},\a_ 1}^m f \big \|_{q,\alpha_ 1}^{\frac{t }{m}}.
\end{equation}
Now, applying the pointwise estimate for $f\in A^{p}_{\a_2}$, we obtain
\begin{displaymath}
\begin{split}
\|S_{\widetilde{\mu},\a_ 1}^m & f\|_{q,\alpha_ 1}^{q}
=\int_{\Bn} \left ((1-|z|^2)^s\int_{\Bn}\frac{|f(w)|^r\,|f(w)|^{m-r}}{|1-\langle z,w \rangle |^{n+1+s+\a_ 1}}
d\widetilde{\mu}(w)\right )^q \,dv_{\a_ 1}(z)
\\
& \le \big \|f \big \|_{p,\a_ 2}^{q(m-r)} \!\int_{\Bn}\!
\left ( \!(1-|z|^2)^s \! \int_{\Bn} \!\!
\frac{|f(w)|^r\, (1-|w|^2)^{A-(m-r)\frac{n+1+\a_ 2}{p}}}{|1-\langle z,w \rangle |^{n+1+s+\a_ 1}}
\,d\mu(w) \right )^q \!dv_{\alpha_ 1}(z)
\\
&=\big \|f \big \|_{p,\a_ 2}^{q(m-r)} \cdot \big \|S_{\mu,\a_ 1}^r f  \big \|^q_{q,\a_ 1}.
\end{split}
\end{displaymath}
Putting this into \eqref{Eq-CII} and using the inequality in part (b), we obtain
\begin{displaymath}
\begin{split}
\big \|T_{\widetilde{\mu}}^{\beta} f \big \|_{t,\sigma}
&\le \big \|B_{s,\gamma ^*}(\widetilde{\mu}) \big \|_{\frac{1}{1-\lambda},\gamma ^*}^{\frac{1}{m'}}\cdot
\big \|f \big \|_{p,\a_ 2}^{\frac{(m-r)}{m}} \cdot \big \|S_{\mu,\a_ 1}^r f \big \|^{1/m}_{q,\a_ 1}
\\
& \le K^{1/m}\cdot \big \|B_{s,\gamma ^*}(\widetilde{\mu})
\big \|_{\frac{1}{1-\lambda},\gamma ^*}^{\frac{1}{m'}}\cdot \big \|f \big \|_{p,\a_ 2}
\end{split}
\end{displaymath}
This together with \eqref{EqKL-3} yield
$$
\|B_{s,\gamma ^*}(\widetilde{\mu})\|_{\frac{1}{1-\lambda},\gamma ^*}\lesssim K^{1/m}\cdot \big \|B_{s,\gamma ^*}(\widetilde{\mu}) \big \|_{\frac{1}{1-\lambda},\gamma ^*}^{\frac{1}{m'}}
$$
which proves that
\begin{equation}\label{Eq-F4}
\|B_{s,\gamma ^*}(\widetilde{\mu})\|_{\frac{1}{1-\lambda},\gamma ^*}\lesssim K
\end{equation}
for $\mu$ with compact support on $\Bn$.
Then a standard limit argument gives \eqref{Eq-F4} for a general positive measure $\mu$.

Now, let $\{a_ k\}$ be any lattice in $\Bn$. Since $(n+1+\gamma ^*) \l -A=(n+1+\gamma) \l $,
applying Theorem \ref{TB} it follows that
\begin{displaymath}
\begin{split}
\|\mu\|_{\l,\g}^{1/(1-\l)}
&\lesssim \sum_ k \left (\frac{\mu(D_ k)}{(1-|a_ k|)^{(n+1+\gamma ^*)\lambda -A}}\right )^{\frac{1}{1-\l }}
\\
&\asymp \sum_ k \left (\frac{\widetilde{\mu}(D_ k)}{(1-|a_ k|)^{(n+1+\gamma ^*)\lambda}}\right )^{\frac{1}{1-\l }}
\\
&\lesssim \|\widetilde{\mu}\|_{\l,\g ^*}^{1/(1-\l)}
\lesssim  \big \|B_{s,\gamma ^*}(\widetilde{\mu})\big \|_{\frac{1}{1-\lambda},\gamma ^*}^{1/(1-\l)}.
\end{split}
\end{displaymath}
Then, from \eqref{Eq-F4} we have that $\mu$ is a $(\l,\g)$-Bergman Carleson measure with
$$\|\mu\|_{\l,\g} \lesssim K.$$
The proof is complete.
\end{proof}

\subsection{Proof of Theorem \ref{carleson}}
The following result together with Proposition \ref{Car-P1} concludes the proof of Theorem \ref{carleson}.
\begin{prop}\label{PNT1}
Let $\lambda>0$. If \eqref{eq1} holds, then $\mu$ is a $(\lambda,\gamma)$-Bergman Carleson measure. Furthermore, $\|\mu\|_{\l,\gamma}\lesssim C$, where $C$ is the constant appearing in \eqref{eq1}.
\end{prop}

\begin{proof}
Assume first that $\l\ge 1$.
Let
$$
f_{i,a}(z)=\frac{(1-|a|^2)^{(n+1+\a_i)/p_i}}{(1-\langle z,a\rangle)^{2(n+1+\a_i)/p_i}}.
$$
Then it can be easily checked that
for every $a\in \Bn$ and for all $i=1,2,\dots,k$, $\|f_{i,a}\|_{p_i,\a_i}\lesssim 1$.
Thus \eqref{eq1} implies
\begin{equation}\label{eq4}
\int_{\Bn}\prod_{i=1}^k |f_{i,a}(z)|\,d\mu(z)
\le C\prod_{i=1}^k\|f_{i,a}\|_{p_i,\a_i}^{q_i}=C,
\end{equation}
where $C$ is a positive constant independent of $a$.
An easy computation shows that
$$
\sum_{i=1}^k(n+1+\a_i)\,\frac{q_i}{p_i}=(n+1+\g)\l.
$$
Thus (\ref{eq4}) is equivalent to
$$
\int_{\Bn}\frac{(1-|a|^2)^{(n+1+\g)\l}}{|1-\langle z,a\rangle|^{2(n+1+\g)\l}}\,d\mu(z)\le C.
$$
Since $\l\ge1$, by Theorem A,
we know that $\mu$ is a $(\l,\g)$-Bergman Carleson measure with $\|\mu\|_{\l,\gamma}\lesssim C$.\\

Next, we consider the case $0<\lambda<1$. We use induction on $k$.
If $k=1$ then \eqref{eq1} is just the definition of a Bergman Carleson measure.
Now, let $k\ge 2$ and assume that the result holds for $k-1$ functions.
Set $\lambda_k=\lambda$; $\gamma_k=\gamma$ and
$$
\lambda_{k-1}=\sum_{i=1}^{k-1}\frac{q_i}{p_i}; \qquad
\g _{k-1}=\frac1{\l _{k-1}}\sum_{i=1}^{k-1}\frac{\a _iq_i}{p_i}.
$$
Considering the measure
$$
d\mu_k(z)=|f_k(z)|^{q_k} \,d\mu(z)
$$
we see that our condition
$$
\int_{\Bn} \prod_{i=1}^{k}|f_i(z)|^{q_i}\,d\mu(z)
\le C \prod_{i=1}^{k}\|f_i\|_{p_i,\a_i}^{q_i}
$$
is equivalent to the condition
$$
\int_{\Bn} \prod_{i=1}^{k-1}|f_i(z)|^{q_i}\,d\mu_k(z)
\le C(f_k)\, \prod_{i=1}^{k-1}\|f_i\|_{p_i,\a_i}^{q_i}
$$
with $C(f_k)=C\cdot \|f_k\|_{p_k,\a_k}^{q_k}$.
By induction, this implies that $\mu_k$ is a $(\lambda_{k-1},\gamma_{k-1})$-Bergman Carleson measure
with $\|\mu_k\|_{\l_{k-1},\gamma_{k-1}}\lesssim C(f_k)$.
Since $0<\l_{k-1}<\l<1$, then Theorem B implies that $B_{s,\gamma_{k-1}}(\mu_k)$
belongs to $L^{1/(1-\lambda_{k-1}),\gamma_{k-1}}$ for any $s>0$ with
$$
\big \|B_{s,\gamma_{k-1}}(\mu_k)\big \|_{\frac{1}{1-\lambda_{k-1}},\gamma_{k-1}} \lesssim C(f_k).
$$
That is, we have
\begin{equation*}
\int_{\Bn} \left ( \int_{\Bn} \frac{(1-|z|^2)^s
|f_k(w)|^{q_k} d\mu(w)}{|1-\langle z,w \rangle |^{n+1+s+\gamma_{k-1}}}\right )^{\frac{1}{1-\l_{k-1}}} dv_{\gamma_{k-1}}(z)
\lesssim \Big (C\cdot \|f_k\|_{p_k,\alpha_k}^{q_k}\Big)^{\frac{1}{1-\l_{k-1}}},
\end{equation*}
or equivalently,
\begin{equation*}
\|S^{q_k}_{\mu,\g _{k-1}} f_k \|_{\frac{1}{1-\l_{k-1}},\gamma_{k-1}}
\lesssim C\cdot \|f_k\|_{p_k,\alpha_k}^{q_k}
\end{equation*}
whenever $f_k$ is in $A^{p_k}_{\a_k}$. Thus, by Lemma \ref{ML},
the measure $\mu$ is an $(\l^*,\g^*)$-Bergman Carleson measure with $\|\mu\|_{\l ^*,\g ^*}\lesssim C$, where
$$
\l^*=1+\frac{q_k}{p_k}-(1-\l_{k-1}),\quad \textrm{and}
\quad \gamma^* =\frac{1}{\lambda^*}\Big(\g_{k-1}+\frac{\alpha_k q_k}{p_k}-\g_{k-1}(1-\l _{k-1}) \Big).
$$
Simple algebraic manipulations shows that $\l^*=\l$ and $\g^*=\g$ concluding the proof.
\end{proof}

\section{Vanishing $(\l,\g)$-Bergman Carleson measures}
\label{v-b-carleson}

We say that $\mu$ is a vanishing $(\lambda,\a)$-Bergman Carleson measure
if for any two positive numbers $p$ and $q$ satisfying $q/p=\lambda$
and any sequence $\{f_k\}$ in $A^p_{\a}$ with $\|f_k\|_{p,\a}\le 1$ and $f_k(z)\to0$
uniformly on any compact subset of $\Bn$,
$$
\lim_{k\to\infty}\int_{\Bn} |f_k(z)|^q\,d\mu(z)=0.
$$
It is well-known that, for $\lambda\ge 1$,
$\mu$ is a vanishing $(\lambda,\a)$-Bergman Carleson measure
if and only if
\begin{equation}\label{v-carleson0}
\lim_{|a|\to1}\int_{\Bn}\frac{(1-|a|^2)^t}{|1-\langle z, a \rangle |^{{(n+1+\a)\l}+t}}\,d\mu(z)=0
\end{equation}
for some (any) $t>0$.
It is also well-known that, for $0<\l<1$, $\mu$ is a vanishing
$(\lambda,\a)$-Bergman Carleson measure if and only if
it is a $(\lambda,\a)$-Bergman Carleson measure.
We refer to \cite{zz} for the above facts.

\begin{thm}\label{v-carleson}
Let $\mu$ be a positive Borel measure on $\Bn$.
For any integer $k\ge 1$ and $i=1,2,...,k$,
let $0<p_i,q_i<\infty$ and $-1<\a_i<\infty$.
Let
$$
\l=\sum_{i=1}^k\frac{q_i}{p_i}; \qquad
\g=\frac1{\l}\sum_{i=1}^k\frac{\a _iq_i}{p_i}.
$$
Then the following statements are equivalent.
\begin{itemize}
\item[(i)] $\mu$ is a vanishing $(\l,\g)$-Bergman Carleson measure.
\item[(ii)] For any sequence $\{f_{1,l}\}$ in the unit ball of $A^{p_1}_{\a_1}$
which is convergent to $0$ uniformly in compact subsets of $\Bn$,
$$
\lim_{l\to\infty}\sup_{f_i\in A^{p_i}_{\a_i}, \|f_i\|_{p_i,\a_i}\le1, i=2,3,...,k}
\int_{\Bn}|f_{1,l}(z)|^{q_1}|f_2(z)|^{q_2}\cdots|f_k(z)|^{q_k}\,d\mu(z)=0.
$$
\item[(iii)] For any $k$ sequences $\{f_{1,l}\}$, $\{f_{2,l}\}$, ..., $\{f_{k,l}\}$
in the unit balls of $A^{p_1}_{\a_1}$, $A^{p_2}_{\a_2}$, ..., $A^{p_k}_{\a_k}$, respectively,
which are all convergent to $0$ uniformly in compact subsets of $\Bn$,
$$
\lim_{l\to\infty}
\int_{\Bn}|f_{1,l}(z)|^{q_1}|f_{2,l}(z)|^{q_2}\cdots|f_{k,l}(z)|^{q_k}\,d\mu(z)=0.
$$
\end{itemize}
\end{thm}

\begin{proof} By the remark preceding the statement of the Theorem,
the case $0<\l<1$ is just a consequence of Theorem \ref{carleson}.
So, we assume that $\l \ge 1$.
Let (i) be true, so $\mu$ is a vanishing $(\l,\g)$-Bergman Carleson measure.
Let $\{f_{1,l}\}$ be a sequence in the unit ball of $A^{p_1}_{\a_1}$
which is convergent to $0$ uniformly in compact subsets of $\Bn$,
and let $\{f_i\}$ be arbitrary functions in the unit balls of $A^{p_i}_{\a_i}$,
$i=2,3,...k$.

Let $\mu_r=\mu|_{\Bn\setminus \overline D_r}$, where $D_r=\{z\in\Bn\,:\,|z|<r\}$.
Then $\mu_r$ is also a $(\l,\g)$-Bergman Carleson measure, and
$$
\lim_{r\to1}\|\mu_r\|_{\l,\g}=0.
$$
(See, p.130 of \cite{cm}.)
Hence
\begin{eqnarray}\label{eq6}
&~&\int_{\Bn\setminus \overline{D}_r} |f_{1,l}(z)|^{q_1}|f_2(z)|^{q_2}
\cdots|f_k(z)|^{q_k}\,d\mu(z)
\\
&~&\qquad\le\int_{\Bn}|f_{1,l}(z)|^{q_1}|f_2(z)|^{q_2}
\cdots|f_k(z)|^{q_k}\,d\mu_r(z) \notag
\\
&~&\qquad\le C\|\mu_r\|_{\l,\g}
\le C\e, \notag
\end{eqnarray}
as $r$ sufficiently close to $1$. Fix such an $r$. Since
$\{f_{1,l}\}$ converges to $0$ uniformly in compact subsets of $\Bn$,
there is a constant $K>0$ such that for any $l>K$, $|f_{1,l}(z)|<\e$ for any
$z\in \overline D_r$. Therefore, using Theorem \ref{carleson},
\begin{eqnarray}\label{eq7}
&~&\int_{\overline D_r} |f_{1,l}(z)|^{q_1}|f_2(z)|^{q_2}\cdots|f_k(z)|^{q_k}\,d\mu(z)
\\
&~&\qquad\le \e\int_{\Bn} |f_2(z)|^{q_2}\cdots|f_k(z)|^{q_k}\,d\mu(z) \notag
\\
&~&\qquad=\e\int_{\Bn} |1|^{q_1} |f_2(z)|^{q_2}\cdots|f_k(z)|^{q_k}\,d\mu(z) \notag
\\
&~&\qquad\lesssim\e\|1\|_{p_1,\a_1}^{p_1}\|f_2\|_{p_2,\a_2}^{p_2}\cdots\|f_k\|_{p_k,\a_k}^{p_k}
\lesssim\e, \notag
\end{eqnarray}
for any $z\in\overline D_r$. Combining (\ref{eq6}) and (\ref{eq7}) we get (ii).

It is obvious that (ii) implies (iii). Now let (iii) be true.
Let
$$
f_{i,a}(z)=\frac{(1-|a|^2)^{(n+1+\a_i)/p_i}}{(1-\langle z,a\rangle)^{2(n+1+\a_i)/p_i}}.
$$
Then, as before, we know that
for every $a\in \Bn$ and for all $i=1,2,...,k$, $\|f_{i,a}\|_{p_i,\a_i}\lesssim 1$
and it can be easily checked that
$$
\lim_{|a|\to1}|f_{i,a}(z)|=0
$$
uniformly on any compact subset of $\Bn$.
Thus (iii) implies
$$
\lim_{|a|\to1}\int_{\Bn}\prod_{i=1}^k
\frac{(1-|a|^2)^{(n+1+\a_i)q_i/p_i}}{|1-\langle z,a\rangle|^{2(n+1+\a_i)q_i/p_i}}\,d\mu(z)=0.
$$
Since $\sum_{i=1}^k(n+1+\a_i)q_i/p_i=(n+1+\g){\l}$,
the above equality is the same as
$$
\lim_{|a|\to1}\int_{\Bn}
\frac{(1-|a|^2)^{(n+1+\g)\l}}{|1-\langle z,a\rangle|^{2(n+1+\g)\l}}\,d\mu(z)=0.
$$
Thus by (\ref{v-carleson0}), $\mu$ is a vanishing $(\l,\g)$-Bergman Carleson measure.
The proof is complete.
\end{proof}

Vanishing Bergman Carleson measures are also useful in order to describe the compactness of Toeplitz operators between weighted Bergman spaces.
\begin{thm}
Let $\mu$ be a positive Borel measure on $\Bn$, $0<p_1, p_2<\infty$, and $-1<\a_1,\a_2<\infty$. Let $\b,\l$ and $\g$ be as in Theorem \ref{toeplitz1}. Then
$T_{\mu}^{\b}$ is compact from $A^{p_1}_{\a_1}$ to $A^{p_2}_{\a_2}$ if and only if $\mu$ is a vanishing $(\l,\g)$-Bergman Carleson measure.
\end{thm}
\begin{proof}
If $0<\l<1$ then by the remark preceding Theorem \ref{v-carleson}, a vanishing $(\l,\g)$-Bergman Carleson measure is the same as a $(\l,\g)$-Bergman Carleson measure. Also, since $0<\l\le 1$, then $0<p_ 2<p_ 1<\infty$, and therefore the result follows from Theorem \ref{toeplitz1} since in that case $T_{\mu}^{\b}$ is compact from $A^{p_1}_{\a_1}$ to $A^{p_2}_{\a_2}$ if and only if it is bounded, due to a general result of Banach space theory:
it is known that, for $0< p_2<p_1<\infty$, every bounded operator
from $\ell^{p_1}$ to $\ell^{p_2}$ is compact
(see, for example, \cite[Theorem I.2.7, p.31]{lt}),
and the weighted Bergman space $A^p_{\a}$ is isomorphic to $\ell^{p}$
(see \cite[Theorem 11, p.89]{wo}, note that the same proof there works
for weighted Bergman spaces on the unit ball $\Bn$).

Next we consider the case $\l\ge 1$. If $T^{\beta}_{\mu}$ is compact, then $\|T^{\beta}_{\mu} f_ k\|_{p_ 2,\a _ 2}\rightarrow 0$ for any bounded sequence $\{f_ k\}$ in $A^{p_ 1}_{\a_ 1}$ converging to zero uniformly on compact subsets of $\Bn$. Let $\{a_ k\}\subset \Bn$ with $|a_ k| \rightarrow 1^{-}$ and consider the functions
$$f_ k(z)=\frac{(1-|a_ k|^2)^{(n+1+\beta)-(n+1+\a _ 1)/p_ 1}}{(1-\langle z,a_ k \rangle )^{n+1+\beta}}.$$
Due to the conditions on $\beta$ and Lemma \ref{Ict} we have $\sup_ k \|f_ k\|_{p_ 1,\a _ 1}<\infty$, and it is obvious that $f_ k$ converges to zero uniformly on compact subsets of $\Bn$. Hence, $\|T^{\beta}_{\mu} f_ k\|_{p_ 2,\a _ 2}\rightarrow 0$. Therefore, proceeding as in the proof of the case $\l \ge 1$ of the implication (i) imples (ii) in Theorem \ref{toeplitz1}, for any $r>0$, we get
\begin{displaymath}
\begin{split}
\frac{\mu \big (D(a_ k,r) \big ) }{(1-|a_ k|^2)^{(n+1+\g)\l}}&\lesssim (1-|a_ k|^2)^{(n+1+\beta)+(n+1+\a _ 1)/p_ 1-(n+1+\g)\l}\,\,T_{\mu}^{\beta} f_ k(a_ k)
\\
&=(1-|a_ k|^2)^{(n+1+\a _ 2)/p_ 2}\,\,T_{\mu}^{\beta} f_ k(a_ k)
\\
&\lesssim \|T_{\mu}^{\beta} f_ k\|_{p_ 2,\a_ 2}\rightarrow 0.
\end{split}
\end{displaymath}
Thus, by \cite[p. 71]{zz}, the measure $\mu$ is a vanishing $(\l,\g)$-Bergman Carleson measure.

Conversely, let $\mu$ be a vanishing $(\l,\g)$-Bergman Carleson measure with $\l \ge 1$. To prove that $T_{\mu}^{\beta}$ is compact, we must show that $\|T^{\beta}_{\mu} f_ k\|_{p_ 2,\a _ 2}\rightarrow 0$ for any bounded sequence $\{f_ k\}$ in $A^{p_ 1}_{\a_ 1}$ converging to zero uniformly on compact subsets of $\Bn$. If $p_ 2>1$ then, as in the proof of Theorem \ref{toeplitz1}, by duality and Theorem \ref{v-carleson} we have (the numbers $p'_ 2$ and $\a '_ 2$ are the ones defined by \eqref{Eq-DN})
\begin{displaymath}
\|T^{\beta}_{\mu} f_ k\|_{p_ 2,\a _ 2}\asymp \sup_ {\|h\|_{p'_ 2,\a '_ 2}\le 1} \big | \langle h, T^{\beta}_{\mu} f_ k \rangle _{\beta} \big |\le \sup _ {\|h\|_{p'_ 2,\a '_ 2}\le 1} \int_{\Bn} |f_ k(z)|\,|h(z)|\,d\mu(z)\rightarrow 0.
\end{displaymath}
If $0<p_ 2 \le 1$, from the estimates obtained in the proof of (ii) implies (i) in Theorem \ref{toeplitz1} (see \eqref{t-mu2}) it follows that, for any lattice $\{a_ j\}$, we have
\begin{equation}\label{Eq-TC-1}
\begin{split}
\|T_{\mu}^{\beta}f_ k\|_{p_2,\alpha_2}^{p_2}
&\lesssim \sum_{j=1}^{\infty}
\left (\frac{\mu(D_j)}{(1-|a_j|^2)^{(n+1+\gamma )\l }}\right )^{p_ 2}
\left(\int_{\tilde D_j}  \!\!  |f_ k(z)|^{p_1}\,dv_{\alpha_1}(z)\right)^{p_2/p_1}.
\end{split}
\end{equation}
Let  $\varepsilon >0$. Since  $\mu$ is a vanishing $(\l,\g)$-Bergman Carleson measure, due to \cite[p. 71]{zz}, there is $0<r_ 0<1$ such that
\begin{equation}\label{Eq-TC-2}
\sup _{|a_ j|>r_ 0} \frac{\mu(D_j)}{(1-|a_j|^2)^{(n+1+\gamma )\l }} <\varepsilon.
\end{equation}
Split the sum appearing in \eqref{Eq-TC-1} in two parts: one over the points $a_ j$ with $|a_ j|\le r_ 0$, and the other over the points with $|a_ j|>r_ 0$. Since $\{f_ k\}$ converges to zero uniformly on compact subsets of $\Bn$, it is clear that the sum over the points $a_ j$ with $|a_ j|\le r_ 0$ (a finite sum) goes to zero as $k$ goes to infinity. On the other hand, by \eqref{Eq-TC-2} and since $p_ 2\ge p_ 1$ (because $\l\ge 1$), we have
\begin{displaymath}
\begin{split}
\sum_{j: |a_ j|>r_ 0} &
\left (\frac{\mu(D_j)}{(1-|a_j|^2)^{(n+1+\gamma )\l }}\right )^{p_ 2}
\left(\int_{\tilde D_j}  \!\!  |f_ k(z)|^{p_1}\,dv_{\alpha_1}(z)\right)^{p_2/p_1}
\\
&<\varepsilon ^{p_ 2} \sum_{j: |a_ j|>r_ 0}
\left(\int_{\tilde D_j}  \!\!  |f_ k(z)|^{p_1}\,dv_{\alpha_1}(z)\right)^{p_2/p_1} \le \varepsilon ^{p _ 2}\|f_ k\|_{p_ 1,\a _ 1}^{p_ 2}\le C \varepsilon ^{p _ 2}.
\end{split}
\end{displaymath}
Thus $\|T^{\beta}_{\mu} f_ k\|_{p_ 2,\a _ 2}\rightarrow 0$, finishing the proof.
\end{proof}

\section{Applications}

As a direct consequence of Theorem~\ref{carleson} and Theorem~\ref{v-carleson}
we have the following result.

\begin{cor}\label{carleson2}
Let $\mu$ be a positive Borel measure on $\Bn$.
Let $p,q>0$, $s\ge0$ and $\a,\delta>-1$ be given constants such that
$q/p+s/(n+1+\delta)\ge 1$.
Let
$$
\lambda=\frac{q}{p}+\frac{s}{(n+1+\delta)}\quad
\textrm{and} \quad
\gamma=\frac{1}{\l}\left(\frac{\alpha q}p+\frac{\delta s}{n+1+\delta}\right).
$$
Then $\mu$ is a $(\lambda, \gamma)$-Bergman Carleson measure
if and only if for any $f\in A^p_{\a}$,
and for some (any) $t>0$,
\begin{equation}\label{eq-c2}
\sup_{a\in \Bn}\int_{\Bn}|f(z)|^q\frac{(1-|a|^2)^t}{|1-\langle z, a\rangle |^{s+t}}\,d\mu(z)
\lesssim \|f\|_{p,\a}^q;
\end{equation}
and $\mu$ is a vanishing $(\lambda, \gamma)$-Bergman Carleson measure
if and only if for some (any) $t>0$, and for any sequence $\{f_k\}$ in $A^p_{\a}$
with $\|f_k\|_{p,\a}\le 1$ and $f_k(z)\to0$
uniformly on any compact subset of $\Bn$,
\begin{equation}\label{eq-c2v}
\lim_{k\to\infty}\sup_{a\in\Bn}
\int_{\Bn}|f_k(z)|^q\frac{(1-|a|^2)^t}{|1-\langle z, a\rangle |^{s+t}}\,d\mu(z)=0.
\end{equation}
\end{cor}

\textbf{Remark.} Note that (\ref{eq-c2}) does not depend on $\delta$,
which means that, in this corollary, we can choose any real number
$\delta>-1$ satisfying $q/p+s/(n+1+\delta)\ge1$
for $\lambda$ and $\gamma$. Furthermore, if $s=0$ we can also take $t=0$ since then the result reduces to the definition of (vanishing) Bergman Carleson measures.

\begin{proof} We begin with the first part. The case $s=0$ follows directly from the definition of Bergman Carleson measures.
So we assume $s>0$.
Since $s/(n+1+\delta)>0$, we can choose two positive numbers $p_2$ and $q_2$
such that $s/(n+1+\delta)=q_2/p_2$. Then
$$
\lambda=\frac{q}{p}+\frac{s}{n+1+\delta}=\frac{q}{p}+\frac{q_2}{p_2}\ge 1,
$$
and
$$
\gamma=\frac1{\lambda}\left(\frac{\a q}{p}+\frac{\delta s}{n+1+\delta}\right)
=\frac1{\lambda}\left(\frac{\a q}{p}+\frac{\delta q_2}{p_2}\right).
$$
Let $\mu$ be an $(\lambda,\gamma)$-Bergman Carleson measure.
Then, from the above observation and Theorem~\ref{carleson},
we know that for any $f\in A^{p}_{\a}$ and $g\in A^{p_2}_{\delta}$,
we have
\begin{equation}\label{eq-c2-g}
\int_{\Bn}|f(z)|^q \,|g(z)|^{q_2}\,d\mu(z)\lesssim \|f\|_{p,\a}^{q}\cdot \|g\|_{p_2,\delta}^{q_2}.
\end{equation}
For any $t>0$, let
$$
g(z)=g_a(z)=\frac{(1-|a|^2)^{t/{q_2}}}{ (1-\langle z,a\rangle)^{(n+1+\delta)/{p_2}+t/{q_2}}}.
$$
Using Lemma \ref{Ict} it is easy to check that $g_a\in A^{p_2}_{\delta}$,
and $\sup_{a\in \Bn}\|g_a\|_{p_2,\delta}\lesssim  1$.
Put $g=g_a$ in equation (\ref{eq-c2-g}),
take supremum over all $a\in \Bn$, and we get (\ref{eq-c2}).

Conversely, suppose (\ref{eq-c2}) holds for some $t>0$.
Given an arbitrary $t_1>0$, let
$$
f_a(z)=\frac{(1-|a|^2)^{t_1/q}}{(1-\langle z, a\rangle )^{(n+1+\a)/p+t_1/q}}.
$$
As before, it is easy to check that
$f_a\in A^p_{\alpha}$ and $\|f_a\|_{p,\alpha}\lesssim 1$.
 It is clear that $$(n+1+\gamma)\lambda=s+(n+1+\a)q/p,$$
and therefore, due to (\ref{eq-c2}) we get
\begin{displaymath}
\int_{\Bn}\frac{(1-|a|^2)^{t+t_1}}{|1-\langle z, a \rangle |^{(n+1+\g)\l+t+t_1}}\,d\mu(z)
=\int_{\Bn}|f_a(z)|^q\frac{(1-|a|^2)^t}{|1-\langle z, a \rangle |^{s+t}}\,d\mu(z)
\lesssim 1.
\end{displaymath}
Therefore
$$
\sup_{a\in \Bn}\int_{\Bn}\frac{(1-|a|^2)^{t+t_1}}{|1-\langle z, a \rangle |^{(n+1+\g)\l+t+t_1}}\,d\mu(z)<\infty,
$$
Since $\l \ge 1$, by Theorem \ref{TA}, we see that $\mu$ is an
$(\lambda,\gamma)$-Bergman Carleson measure.\\

Next we deal with the part concerning vanishing Bergman Carleson measures. If $s=0$, then the result follows easily from the definition of vanishing Bergman Carleson measures, and so we assume that $s>0$. If $\mu$ is a vanishing $(\lambda,\gamma)$-Bergman Carleson measure then, proceeding as in the first part, but using Theorem \ref{v-carleson} instead of Theorem \ref{carleson}, we obtain \eqref{eq-c2v}. Conversely, suppose that \eqref{eq-c2v} holds for some $t>0$. Let $\{a_ k\}\subset \Bn$ with $|a_ k|\rightarrow 1$ and, for arbitrary $t_ 1 >0$, consider the functions
$$f_ k(z)=\frac{(1-|a_ k|^2)^{t_1/q}}{(1-\langle z, a _ k\rangle )^{(n+1+\a)/p+t_1/q}}.$$
Then $\sup _ k \|f_ k\|_{p,\a} \le C$ and $\{f_ k\}$ converges to zero uniformly on compact subsets of $\Bn$, and using \eqref{eq-c2v} we see that
\begin{displaymath}
\begin{split}
\int_{\Bn} \frac{(1-|a_ k|^2)^{t+t_ 1}}{|1-\langle z, a_ k\rangle |^{(n+1+\g)\l +t+t_ 1}} & \,d\mu(z)=\int_{\Bn} |f_ k(z)|^q \,\frac{(1-|a_ k|^2)^t}{|1-\langle z,a_ k\rangle |^{s+t}}\,d\mu(z)
\\
&\le \sup_{a\in \Bn} \int_{\Bn} |f_ k(z)|^q \,\frac{(1-|a|^2)^t}{|1-\langle z,a\rangle |^{s+t}}\,d\mu(z)\longrightarrow 0.
\end{split}
\end{displaymath}
Since $\l \ge 1$, it follows from \eqref{v-carleson0} that $\mu$ is a vanishing $(\lambda,\gamma)$-Bergman Carleson measure. The proof is complete.
\end{proof}

\subsection{Applications to extended Ces\`aro operators}
\label{Cesaro}

For $g\in H(\Bn)$, the radial derivative is defined by
$$
Rg(z)=\sum_{k=1}^nz_k\frac{\partial g}{\partial z_k}(z),
$$
and the extended Ces\`aro operator is defined by
$$
J_gf(z)=\int_0^1 f(tz)Rg(tz)\,\frac{dt}{t},\qquad f\in H(\Bn).
$$
In the case of one variable, the operator is the same as
$$
J_gf(z)=\int_0^z f(\xi)g'(\xi)\,d\xi,
$$
which is also called the Riemann-Stieltjes operator.
The operator $J_g$ was first used by Ch. Pommerenke to
characterize BMOA functions on the unit disk.
It was first systematically studied
by A. Aleman and A. G. Siskasis  in \cite{as1}. They proved that $J_g$
is bounded on the Hardy space $H^p$ on the unit disk if and only if $g\in BMOA$.
Thereafter there have been many works on these operators.
See,  \cite{ac}, \cite{as2}, \cite{h}, \cite{pp1}, \cite{pp2}, \cite{sz} and \cite{x} for a few examples.
Here we are considering boundedness and compactness
of these operators from a
weighted Bergman space into the general space $F(p,q,s)$ on the unit ball, which is
defined as the space of all holomorphic functions $f$ on $\Bn$ such that
$$
\|f\|_{F(p,q,s)}^p
=\sup_{a\in \Bn}\int_{\Bn}|Rf(z)|^p(1-|z|^2)^q(1-|\p_a(z)|^2)^s\,dv(z)<\infty,
$$
where  $0<p<\infty$, $-n-1<q<\infty$, $0\le s<\infty$,
and $q+s>-1$.
We also say $f\in F_0(p,q,s)$ if
$$
\lim_{|a|\to1}\int_{\Bn}|Rf(z)|^p(1-|z|^2)^q(1-|\p_a(z)|^2)^s\,dv(z)=0.
$$

The family of spaces $F(p,q,s)$ on the unit disk was introduced
in \cite{zha1}. It contains, as special cases, many classical function spaces,
such as the analytic Besov spaces, weighted Bergman spaces,
Dirichlet spaces, the Bloch space, BMOA and $Q_p$ spaces.
See \cite{zha1} for the details. For $F(p,q,s)$ on the unit ball,
we refer to \cite{zhc}.

Here, for our purpose, we point out that if $s>n$, and $\a>0$,
then for any $p>0$, the space $F(p,p\a-n-1,s)=B^\a$, the $\a$-Bloch space,
which means the space of all functions $f\in H(\Bn)$ such that
$$
\|f\|_{B^{\a}}=\sup_{z\in \Bn}|Rf(z)|(1-|z|^2)^{\a}<\infty.
$$
When $\a=1$, $B^{1}=B$, the classical Bloch space.

For the case of the unit disk, the above result can be found in \cite{zha1}.
For the case of the unit ball $\Bn$, the result may be also known,
but we were not able to find a reference, so we provide a brief proof here.
First, we show that $F(p,q,s)$ are all subspaces of some $\a$-Bloch space.

\begin{prop}\label{fpqs} Let $0<p<\infty$, $-n-1<q<\infty$, $0\le s<\infty$,
and $q+s>-1$. Then $F(p,q,s)\subseteq B^{(n+1+q)/p}$.
\end{prop}

\begin{proof}
Let $f\in F(p,q,s)$.
By subharmonicity we have that, for a fixed $r$, $0<r<1$,
$$
|Rf(a)|^p\lesssim \frac1{(1-|a|^2)^{n+1}}\int_{D(a,r)}|Rf(z)|^p\,dv(z).
$$
Let $\a=(n+1+q)/p$. Then $q=p\a-n-1$.
Hence
\begin{eqnarray*}
|Rf(a)|^p(1-|a|^2)^{p\a}
&\lesssim& \frac1{(1-|a|^2)^{n+1-p\a}}\int_{D(a,r)}|Rf(z)|^p\,dv(z)\\
&\lesssim& \int_{D(a,r)}|Rf(z)|^p(1-|z|^2)^{p\a-n-1}\,dv(z)\\
\end{eqnarray*}
Since $|1-\langle z,a \rangle | \asymp (1-|z|^2)\asymp (1-|a|^2)$ for $z\in D(a,r)$, we know from \eqref{eq-pa} that $1-|\p_a(z)|^2 \asymp  1$ for $z\in D(a,r)$, and so, for $s\ge 0$,
\begin{eqnarray*}
|Rf(a)|^p(1-|a|^2)^{p\a}
&\lesssim& \int_{D(a,r)}|Rf(z)|^p(1-|z|^2)^{p\a-n-1}(1-|\p_a(z)|^2)^s\,dv(z)\\
&\lesssim& \int_{\Bn}|Rf(z)|^p(1-|z|^2)^{p\a-n-1}(1-|\p_a(z)|^2)^s\,dv(z).
\end{eqnarray*}
This clearly implies $F(p,p\a-n-1,s)\subseteq B^{\a}$, or $F(p,q,s)\subseteq B^{(n+1+q)/p}$.
\end{proof}

\begin{prop}\label{bloch}
Let $0<p<\infty$, $-n-1<q<\infty$, $0\le s<\infty$,
and $q+s>-1$. If $s>n$ then $F(p,q,s)=B^{(n+1+q)/p}$.
\end{prop}

\begin{proof}
Let $\a=(n+1+q)/p$.
The inclusion $F(p,q,s)\subseteq B^{\a}$ has been proved in the previous proposition.
Now we are proving the opposite inclusion.
Let $f\in B^{\a}$, and assume that $s>n$. Then, by Lemma~\ref{Ict}
\begin{eqnarray*}
\|f\|_{F(p,q,s)}^p
&=&\sup_{a\in \Bn}\int_{\Bn}|Rf(z)|^p(1-|z|^2)^{p\a-n-1}(1-|\p_a(z)|^2)^s\,dv(z)\\
&\le&\|f\|_{B^{\a}}^p\sup_{a\in \Bn}\int_{\Bn}(1-|z|^2)^{-n-1}(1-|\p_a(z)|^2)^s\,dv(z)\\
&\lesssim& \|f\|_{B^{\a}}^p\sup_{a\in \Bn}(1-|a|^2)^s
\int_{\Bn}\frac{(1-|z|^2)^{s-n-1}}{|1-\langle z,a\rangle|^{2s}}\,dv(z)\\
&\lesssim&\|f\|_{B^{\a}}^p,
\end{eqnarray*}
and so $f\in F(p,q,s)$. The proof is complete.
\end{proof}

In a similar way we can prove that, under the same restrictions of the parameters,
$F_0(p,q,s)\subseteq B^{(n+1+q)/p}_0$; and $F_0(p,q,s)=B^{(n+1+q)/p}_0$ if $s>n$,
where, for $\a>0$, $B^{\a}_0$ is the closed subspace of $B^{\a}$
which consists of functions $f\in H(\Bn)$ such that
$$
\lim_{|z|\to1}|Rf(z)|(1-|z|^2)^{\a}=0,
$$
and is called the little $\a$-Bloch space. We will frequently use the following well-known result \cite[Exercise 7.7]{ZhuBn} for
the $\a$-Bloch space: for $\a>1$, an analytic function $f\in B^{\a}$
if and only if
$$
\sup_{z\in\Bn}|f(z)|(1-|z|^2)^{\a-1}<\infty,
$$
and the norm of $f$ in $B^{\a}$ is
\begin{equation}\label{a-bloch}
|f(0)|+\|f\|_{B^{\a}}\asymp\sup_{z\in \Bn}|f(z)|(1-|z|^2)^{\a-1},\qquad \a>1.
\end{equation}

\begin{thm}\label{jg}
Let $0<p,t,\alpha<\infty$, $-1<\beta<\infty$, $0\le s<\infty$,
with $p\beta+s>n$.
Let $g\in H(\Bn)$ and
suppose $\beta-(n+1+\alpha)/t>0$ and
$p/t+s/(n+1+\delta)\ge 1$ for some $\delta>-1$.
Then
\begin{enumerate}
\item [(a)]  $J_g$ is a bounded operator
from $A^t_{\alpha}$ into $F(p,p\beta-n-1,s)$ if and only if
$g\in B^{\beta-(n+1+\a)/t}$;
\item [(b)]
 $J_g$ is a compact operator
from $A^t_{\alpha}$ into $F(p,p\beta-n-1,s)$ if and only if
$g\in B^{\beta-(n+1+\a)/t}_0$.
\end{enumerate}
\end{thm}

\begin{proof}
An easy computation shows that $R (J_gf)=fRg$.
By definition, $J_g$ is bounded
from $A^t_{\alpha}$ into $F(p,p\beta-n-1,s)$ if and only if
for any $f\in A^t_{\alpha}$
\begin{eqnarray*}
\|J_gf\|_{F(p,p\beta-n-1,s)}^p
&=&\sup_{a\in \Bn}\int_{\Bn}|f(z)|^p|Rg(z)|^p(1-|z|^2)^{p\beta-n-1}(1-|\p_a(z)|^2)^s\,dv(z)
\\
&=&\sup_{a\in \Bn}\int_{\Bn}|f(z)|^p|Rg(z)|^p(1-|z|^2)^{s+p\beta-n-1}\frac{(1-|a|^2)^s}{|1-\langle z,a\rangle|^{2s}}\,dv(z)
\\
&=&\sup_{a\in \Bn}\int_{\Bn}|f(z)|^p\frac{(1-|a|^2)^s}{|1-\langle z,a\rangle|^{2s}}\,d\mu_g(z)
\\
&\le& C\|f\|_{t,\alpha}^p,
\end{eqnarray*}
where $d\mu_g(z)=|Rg(z)|^p(1-|z|^2)^{s+p\beta-n-1}\,dv(z)$.
By Corollary~\ref{carleson2}, this is equivalent to that
$\mu_g$ is an $(\lambda,\gamma)$-Bergman Carleson measure,
where
$$
\lambda=\frac{p}{t}+\frac{s}{n+1+\delta},\quad \textrm{and}\quad
\gamma=\frac1{\lambda}\left(\frac{\a p}{t}+\frac{\delta s}{n+1+\delta}\right).
$$
Then, by the condition in the theorem, $\lambda\ge 1$ and,
it is easy to check that $\gamma>-1$.
Thus, by Theorem \ref{TA}, the boundedness of $J_ g$ is equivalent to
\begin{equation}\label{EqJg}
\sup_{a\in \Bn}\int_{\Bn}
\frac{(1-|a|^2)^{(n+1+\gamma)\lambda}}{|1-\langle z,a\rangle|^{2(n+1+\gamma)\lambda}}
|Rg(z)|^p(1-|z|^2)^{s+p\beta-n-1}\,dv(z)<\infty.
\end{equation}
An easy computation shows that
$$
(n+1+\gamma)\lambda
%%=[n+1+\alpha p/(\lambda t)]\lambda=(n+1)\lambda+\alpha p/t
%%=(n+1)(s/(n+1)+p/t)+\alpha p/t
=s+(n+1+\alpha)p/t,
$$
and \eqref{EqJg} becomes
$$
\sup_{a\in \Bn}\int_{\Bn}
\frac{(1-|a|^2)^{s+(n+1+\alpha)p/t}}{|1-\langle z,a\rangle|^{2(s+(n+1+\alpha)p/t)}}
|Rg(z)|^p(1-|z|^2)^{s+p\beta-n-1}\,dv(z)<\infty,
$$
which is the same as
$$
\sup_{a\in \Bn}\int_{\Bn}
|Rg(z)|^p(1-|z|^2)^{p(\beta-(n+1+\a)/t)-n-1}(1-|\p_a(z)|^2)^{s+(n+1+\a)p/t}\,dv(z)
<\infty.
$$
Thus the operator $J_ g$ is bounded from $A^t_{\alpha}$ into $F(p,p\beta-n-1,s)$ if and only if
\begin{equation} \label{EqJg-2}
g\in F \Big (p, q, s+(n+1+\a)p/t \Big )
\end{equation}
with
$$ q=p(\beta-(n+1+\a)/t)-n-1.$$
Since $\lambda\ge 1$ and $\gamma>-1$, we know that
$$
s+(n+1+\a)p/t=(n+1+\gamma)\lambda\ge n+1+\gamma>n,
$$
and so, by Proposition~\ref{bloch}, condition \eqref{EqJg-2} is equivalent to
$g\in B^{\beta-(n+1+\a)/t}$ which proves part (a).

Using the second part of Corollary~\ref{carleson2},
the criterion for the compactness in part (b) is proved in the same way.
We omit the details.
The proof is complete.
\end{proof}

%\textit{Remark.} For the case $0<\beta\le (n+1+\a)/t$, the above proof also works.
%If $\beta=(n+1+\a)/t$ then $F(p, p(\beta-(n+1+\a)/t)-(n+1), s+(n+1+\alpha)p/t)=F(p,-(n+1),s+(n+1+\alpha)p/t)$,
%which is the same as the space of analytic functions $g$ such that
%$Rg\in H^{\infty}$ (see \cite[Lemma 2]{zha2}).
%If $0<\beta<(n+1+\a)/t$ then $F(p, p(\beta-(n+1+\a)/t)-(n+1), s+(n+1+\alpha)p/t)$ contains only constant functions.

Our next result is for the integral operator
$$
I_gf(z)=\int_0^1 Rf(tz)\,g(tz)\,\frac{dt}{t}.
$$
This operator can be considered as a companion of the operator $J_g$.

\begin{thm}\label{ig}
Let $0<p,t,\beta<\infty$, $-1<\a<\infty$, $0\le s<\infty$ with
$p\beta+s>n$. Let $g\in H(\Bn)$ and
suppose $p/t+s/(n+1+\delta)\ge 1$ for some $\delta>-1$.
Then $I_g$ is a bounded operator
from $A^t_{\alpha}$ into $F(p,p\beta-n-1,s)$ if and only if
\begin{itemize}
\item[ (i)] $g\in B^{\beta-(n+1+\a)/t}$ \,for\, $\beta>1+(n+1+\a)/t$;
\item[(ii)] $g\in H^{\infty}$ \,for\, $\beta=1+(n+1+\a)/t$;
\item[(iii)] $g\equiv 0$ \,for\, $0<\beta<1+(n+1+\a)/t$.
\end{itemize}
\end{thm}

\begin{proof}
Assume that (i)-(iii) hold.
First, we consider case (i), that is, when $\beta>1+(n+1+\a)/t$.
Let $g\in B^{\beta-(n+1+\a)/t}$.
An easy computation shows that $R(I_gf)=gRf$.
Hence, due to \eqref{a-bloch}, for any $f\in A^t_{\alpha}$ we have
\begin{equation}\label{igf}
\begin{split}
\|I_g&f\|_{F(p,p\beta-n-1,s)}^pCase
\\
&=\sup_{a\in\Bn}\int_{\Bn}|g(z)|^p\,|Rf(z)|^p\,(1-|z|^2)^{p\beta-n-1}(1-|\p_a(z)|^2)^s\,dv(z)
\\
&=\sup_{a\in\Bn}\int_{\Bn}|g(z)|^p\,|Rf(z)|^p \,(1-|z|^2)^{s+p\beta-n-1}
\frac{(1-|a|^2)^s}{|1-\langle z,a\rangle|^{2s}}\,dv(z)
\\
&\le \|g\|_{B^{\beta-(n+1+\a)/t}}^p\sup_{a\in \Bn}\int_{\Bn}|Rf(z)|^p
\frac{(1-|a|^2)^s}{|1-\langle z,a\rangle|^{2s}}\,d\mu(z).
\end{split}
\end{equation}
where $d\mu(z)=(1-|z|^2)^{(n+1+\a)p/t+p+s-n-1}\,dv(z)$.
By Lemma~\ref{Ict}, for any $\eta>0$ we have
\begin{eqnarray}\label{mu}
&~&\sup_{a\in\Bn}\int_{\Bn}\frac{(1-|a|^2)^{\eta}}
{|1-\langle z,a\rangle|^{\eta+(n+1+\a)p/t+p+s}}\,d\mu(z)
\\
&~&\qquad=\sup_{a\in\Bn}(1-|a|^2)^{\eta}\int_{\Bn}
\frac{(1-|z|^2)^{(n+1+\a)p/t+p+s-n-1}}
{|1-\langle z,a\rangle|^{\eta+(n+1+\a)p/t+p+s}}\,dv(z) <\infty. \notag
\end{eqnarray}
Notice that the application of Lemma \ref{Ict} here is correct, since if we let
$$
\lambda=\frac{p}{t}+\frac{s}{n+1+\delta},\quad \textrm{and}\quad
\gamma=\frac1{\lambda}\left(\frac{(t+\a)p}{t}+\frac{\delta s}{n+1+\delta}\right),
$$
then, by the condition in the theorem, $\lambda\ge 1$, and
it is easy to check that $\gamma>-1$. Also, an easy computation shows that
\begin{equation}\label{Eq-Z}
(n+1+\gamma)\lambda
=s+(n+1+\alpha+t)p/t=(n+1+\a)p/t+p+s.
\end{equation}
Thus we have
\begin{displaymath}
\begin{split}
(n+1+\a)p/t+p+s-n-1&= (n+1+\gamma)\lambda -n-1
\\
&\ge n+1+\g -n-1=\g>-1.
\end{split}
\end{displaymath}
Hence, due to \eqref{Eq-Z}, condition \eqref{mu} means that $\mu$ is a $(\lambda,\gamma)$-Bergman Carleson measure,
and so by (\ref{igf}), Corollary~\ref{carleson2} and \cite[Theorem 2.16]{ZhuBn} we have that
$$
\|I_gf\|_{F(p,p\beta-n-1,s)}^p
\lesssim \|g\|_{B^{\beta-(n+1+\a)/t}}^p\|Rf\|_{t,t+\a}^{p}
\asymp \|g\|_{B^{\beta-(n+1+\a)/t}}^p\|f\|_{t,\a}^{p},
$$
and so $I_g:\,A^t_{\alpha}\to F(p,p\beta-n-1,s)$ is bounded.

Case (ii) is proved in the exactly same way as the proof for case (i),
with $\|g\|_{B^{\beta-(n+1+\a)/t}}$ replaced by $\|g\|_{H^{\infty}}$.
Case (iii) is trivial.

Conversely, suppose that $I_g:\,A^t_{\alpha}\to F(p,p\beta-n-1,s)$ is bounded. Then, by Proposition \ref{fpqs}, the operator $I_g:\,A^t_{\alpha}\to B^{\beta}$ is also bounded.
For $\eta>0$ and  $a\in \Bn$, let
$$
f_a(z)=\frac{(1-|a|^2)^{\eta}}{(1-\langle z,a\rangle)^{\eta+(n+1+\a)/t}}.
$$
It is easy to check that $\sup_{a\in\Bn} \|f_a\|_{t,\a}\le C$.
An easy computation shows that
$$
Rf_a(z)=\frac1{\eta+(n+1+\a)/t}
\frac{(1-|a|^2)^{\eta}}{(1-\langle z,a\rangle)^{\eta+(n+1+\a)/t+1}}.
$$
Note that $R(I_gf_a)(z)=Rf_a(z)\,g(z)$, and therefore
\begin{equation}\label{Ig-Eq-1}
\begin{split}
(1-|a|^2)^{\beta-(n+1+\a)/t-1} |g(a)|&\asymp |Rf_ a(a)|\,|g(a)|\,(1-|a|^2)^{\beta}
\\
&\le \sup_{z\in\Bn}|Rf_a(z)|\,|g(z)|\,(1-|z|^2)^{\beta}
\\
&=\|I_gf_a\|_{B^{\beta}}\le C\,\|I_ g\|.
\end{split}
\end{equation}
This directly gives (ii) and, by the maximum principle, we also obtain (iii). Part (i) follows from \eqref{a-bloch}.
The proof is complete.
\end{proof}

Similarly, by a standard method, we can prove the following compactness result.
\begin{thm}\label{ig0}
Let $0<p,t,\beta<\infty$, $-1<\a<\infty$, $0\le s<\infty$ with
$p\beta+s>n$. Let $g\in H(\Bn)$ and
suppose $p/t+s/(n+1+\delta)\ge 1$ for some $\delta>-1$.
Then $I_g$ is a compact operator
from $A^t_{\alpha}$ into $F(p,p\beta-n-1,s)$ if and only if
\begin{itemize}
\item[(i)] $g\in B^{\beta-(n+1+\a)/t}_0$ \,for\, $\beta>1+(n+1+\a)/t$;
\item[(ii)] $g\equiv 0$ \,for\, $0<\beta\le 1+(n+1+\a)/t$.
\end{itemize}
\end{thm}
\begin{proof}
If $I_g$ is a compact operator
from $A^t_{\alpha}$ into $F(p,p\beta-n-1,s)$, then $I_ g:A^t_{\a}\rightarrow B^{\beta}$ is also compact due to Proposition \ref{fpqs}. Let $\{a_ k\}\subset \Bn$ with $|a_ k|\rightarrow 1$ and, for $\eta>0$, consider the sequence of holomorphic functions $\{f_ k\}$ given by
$$f_ k(z)= \frac{(1-|a_ k|^2)^{\eta}}{(1-\langle z,a_ k\rangle)^{\eta+(n+1+\a)/t}}.$$
As before, $\sup_ k \|f_ k\|_{t,\a}\le C$, and $\{f_ k\}$ converges to zero uniformly on compact subsets of $\Bn$. Since $I_ g$ is compact, from \eqref{Ig-Eq-1} we get
$$(1-|a_ k|^2)^{\beta-(n+1+\a)/t-1} |g(a_ k)|\lesssim \|I_gf_k\|_{B^{\beta}}\rightarrow 0.$$
This gives (ii) by the maximum principle; and also (i) since, as in \eqref{a-bloch}, for $\sigma>1$, a function $f\in H(\Bn)$ is in $B^{\sigma}_ 0$ if and only if $\lim_{|z|\rightarrow 1^{-}}(1-|z|^2)^{\sigma-1}|f(z)|=0$.

Conversely, assume that (i) holds, that is, $\beta >1+(n+1+\a)/t$ and $g\in B_ 0^{\beta-(n+1+\a)/t}$. Then, given $\varepsilon>0$, there is $0<r_ 0<1$ such that
\begin{equation}\label{Ig-C-1}
\sup_{r_ 0<|z|<1} (1-|z|^2)^{\beta-(n+1+\a)/t-1}\,|g(z)| <\varepsilon.
\end{equation}
Let $\{f_ k\}$ be a bounded sequence in $A^t_{\a}$ converging to zero uniformly on compact subsets of $\Bn$. From \eqref{igf} we get
\begin{displaymath}
\|I_ g f_ k\|^p _{F(p,p\beta-n-1,s)} =I_ 1(k)+I_ 2(k),
\end{displaymath}
with
$$I_ 1(k):=\sup_{a\in\Bn}\int_{|z|\le r_ 0}\!\!|g(z)|^p\,|Rf_ k(z)|^p \,(1-|z|^2)^{s+p\beta-n-1}
\frac{(1-|a|^2)^s}{|1-\langle z,a\rangle|^{2s}}\,dv(z),$$
and
$$I_ 2(k):=\sup_{a\in\Bn}\int_{r_ 0<|z|<1}\!\!|g(z)|^p\,|Rf_ k(z)|^p \,(1-|z|^2)^{s+p\beta-n-1}
\frac{(1-|a|^2)^s}{|1-\langle z,a\rangle|^{2s}}\,dv(z).$$
Since $\{Rf_ k\}$ also converges to zero uniformly on compact subsets of $\Bn$, there is a positive integer $k_ 0$ such that $\sup_{|z|\le r_ 0}|Rf_ k(z)|<\varepsilon$ for $k\ge k_ 0$. Then, using \eqref{a-bloch}, it is easy to see that
\begin{displaymath}
I_ 1 (k)\le C\varepsilon ^p \|g\|_{B^{\beta-(n+1+\a)/t}}^p .
\end{displaymath}
On the other hand, by \eqref{Ig-C-1}, and arguing as in the proof of Theorem \ref{ig} we obtain
\begin{displaymath}
I_ 2 (k)<C \varepsilon ^p \,\|f_ k\|_{t,\a}^{p}\le C \varepsilon ^p.
\end{displaymath}
This shows that $\|I_ g f_ k\| _{F(p,p\beta-n-1,s)}\rightarrow 0$ proving that $I_ g$ is compact. Since the case (ii) is trivial, the proof is complete.
\end{proof}

\subsection{Pointwise multipliers}

For an holomorphic function $g$ in $\Bn$, the pointwise multiplication
operator $M_g$ is defined as follows $M_gf=gf$ for $f\in H(\Bn)$.

\begin{lemma}\label{mg-bloch}
Let $-1<\a<\infty$, $0<t,\beta<\infty$, and suppose that
$M_g:\,A^t_{\a}\to B^{\beta}$ is bounded. Then
\begin{itemize}
\item[(i)] $g\in B^{\beta-(n+1+\a)/t}$ \,if\, $\beta>1+(n+1+\a)/t$;
\item[(ii)] $g\in H^{\infty}$ \,if\, $\beta=1+(n+1+\a)/t$;
\item[(iii)] $g\equiv 0$ \,if\, $0<\beta<1+(n+1+\a)/t$.
\end{itemize}
\end{lemma}

\begin{proof}
By definition it is easy to see that $B^{\beta_1}\subseteq B^{\beta_2}$
for $\beta_1<\beta_2$. Hence, in case (iii) we may assume that
$1<\beta<1+(n+1+\a)/t$.
For $\eta>0$ and $a\in \Bn$, let
$$
f_a(z)=\frac{(1-|a|^2)^{\eta}}{(1-\langle z,a\rangle)^{\eta+(n+1+\a)/t}}.
$$
We have seen before that $\{f_a\}$ is uniformly bounded in $A^t_{\a}$.
Since $M_g:\,A^t_{\a}\to B^{\beta}$ is bounded,
we know that
$$
\sup_{a\in\Bn}(|g(0)f_a(0)|+\|gf_a\|_{B^{\beta}})
\lesssim \sup_{a\in\Bn}\|f_a\|_{t,\a}<\infty.
$$
However, since $\beta>1$, by (\ref{a-bloch}) we get
\begin{eqnarray*}
|g(0)f_a(0)|+\|gf_a\|_{B^{\beta}}
&=&|g(0)|(1-|a|^2)^{\eta}+\sup_{z\in\Bn}|g(z)||f_a(z)|(1-|z|^2)^{\beta-1}\\
&\ge& |g(a)||f_a(a)|(1-|a|^2)^{\beta-1}\\
&=&|g(a)|(1-|a|^2)^{\beta-1-(n+1+\a)/t}.
\end{eqnarray*}
Hence we get
$$
\sup_{a\in\Bn}|g(a)|(1-|a|^2)^{\beta-1-(n+1+\a)/t}<\infty,
$$
which gives (i) and (ii), and also gives (iii) by the maximum principle.
\end{proof}

Now we are ready to prove the following characterizations for
bounded pointwise multiplication operators from $A^t_{\a}$ to $F(p,q,s)$ spaces.

\begin{thm}\label{mg}
Let $0<p,t,\beta<\infty$, $-1<\a<\infty$, $0\le s<\infty$ with
$p\beta+s>n$. Let $g\in H(\Bn)$ and
suppose $p/t+s/(n+1+\delta)\ge 1$ for some $\delta>-1$.
Then $M_g:\, A^t_{\a}\to F(p,p\beta-n-1,s)$ is bounded if and only if
\begin{itemize}
\item[(i)] $g\in B^{\beta-(n+1+\a)/t}$ \,for\, $\beta>1+(n+1+\a)/t$;
\item[(ii)] $g\in H^{\infty}=H^{\infty}\cap B$ \,for\, $\beta=1+(n+1+\a)/t$;
\item[(iii)] $g\equiv 0$ \,for\, $0<\beta<1+(n+1+\a)/t$.
\end{itemize}
\end{thm}

\begin{proof}
Suppose (i)-(iii) hold.
Then, by Theorem~\ref{jg} and Theorem~\ref{ig} we know that
both $J_g$ and $I_g$ are bounded from $A^t_{\a}$ to $F(p,p\beta-n-1,s)$.
Since
\begin{equation*}\label{rmgf}
R(M_gf)=fRg+gRf=R(I_gf)+R(J_gf),
\end{equation*}
we easily see that $M_g$ is also  bounded from $A^t_{\a}$ to $F(p,p\beta-n-1,s)$.

Conversely, suppose that $M_g:\, A^t_{\a}\to F(p,p\beta-n-1,s)$ is bounded. By Proposition \ref{fpqs}, we know that
$M_g: A^t_{\a}\to B^{\beta}$ is also bounded, and so by Lemma~\ref{mg-bloch},
we get (i)-(iii) directly.
The proof is complete.
\end{proof}

The result on
the compactness of the multiplication operator is stated next.

\begin{thm}\label{mg0}
Let $0<p,t,\beta<\infty$, $-1<\a<\infty$, $0\le s<\infty$ with
$p\beta+s>n$. Let $g\in H(\Bn)$ and
suppose $p/t+s/(n+1+\delta)\ge 1$ for some $\delta>-1$.
Then $M_g:\, A^t_{\a}\to F(p,p\beta-n-1,s)$ is compact if and only if
\begin{itemize}
\item[(i)] $g\in B^{\beta-(n+1+\a)/t}_0$ \,for\, $\beta>1+(n+1+\a)/t$;
\item[(iii)] $g\equiv 0$ \,for\, $0<\beta\le 1+(n+1+\a)/t$.
\end{itemize}
\end{thm}
This follows,  using standard arguments, arguing in a similar way as in Theorem \ref{mg}. We omit the proof here.


\begin{thebibliography}{99}

\bibitem{ac}
A. Aleman and J. A. Cima,
{\it An integral operator on $H^p$ and Hardy's inequality},
J. Anal. Math. 85 (2001), 157--176.

\bibitem{as1}
A. Aleman and A. G. Siskakis,
{\it An integral operator on $H^p$},
Complex Variables 28 (1995), 149--158.

\bibitem{as2}
A. Aleman and A. G. Siskakis,
\textit{Integration operators on Bergman spaces},
Indiana Univ. Math. J. 46 (1997) 337--356.

\bibitem{Car0} L. Carleson,
\emph{An interpolation problem for bounded analytic functions},
Amer. J. Math. 80 (1958), 921--930.

\bibitem{Car1} L. Carleson,
\emph{Interpolations by bounded analytic functions and the corona problem},
Ann. of Math. 76 (1962), 547--559.

\bibitem{CKY} B. R. Choe, H. Koo \and H. Yi,
\emph{Positive Toepliz operators between harmonic Bergman spaces},
Potential Anal. 17 (2002), 307--335.

\bibitem{CW} J. Cima \and W. Wogen,
\emph{A Carleson measure theorem for the Bergman space of the ball},
J. Operator Theory 7 (1982), 157--165.

\bibitem{cm}
C. Cowen and B. MacCluer,
`Composition Operators on Spaces of Analytic Functions',
CRC Press, Boca Raton, 1995.

\bibitem{Du} P. Duren, `Theory of $H\sp{p} $ Spaces', Academic Press,
New~York-London 1970. Reprint: Dover, Mineola, New York 2000.

\bibitem{H} W. Hastings,
\textit{A Carleson measure theorem for Bergman spaces}, Proc. Amer. Math. Soc. 52 (1975), 237--241.

\bibitem{h} Z. Hu,
\textit{Extended Ces\`aro operators on mixed norm spaces},
Proc. Amer. Math. Soc. \textit{131} (2003) 2171--2179.

\bibitem{lt} J. Lindenstrauss and L. Tzafriri,
`Classical Banach Spaces',
Lecture Notes in Math. 338,
Springer-Verlag, Berlin, 1973.

\bibitem{l0} D. H. Luecking,
\textit{A technique for characterizing Carleson measures on Bergman spaces},
Proc. Amer. Math. Soc. 87 (1983), 656--660.

\bibitem{l1} D. H. Luecking,
\textit{Forward and reverse Carleson inequalities for functions
in Bergman spaces and their derivative},
Amer. J. Math. 107 (1985), 85--111.

\bibitem{l1-b} D. H. Luecking,
{\it Representations and duality in weighted spaces of analytic functions},
Indiana Univ. Math. J.
34 (1985), 319--336.

\bibitem{l2} D. H. Luecking,
{\it Trace ideal criteria for Toeplitz operators},
J. Funct. Anal. 73 (1987), 345--368.

\bibitem{l3} D. H. Luecking,
{\it Embedding theorems for spaces of analytic functions via Khinchine's inequality},
Michigan Math. J. 40 (1993), 333--358.

\bibitem{Ol} V. Oleinik,
\textit{Imbedding theorems for weighted classes of harmonic and analytic functions},
Investigations on linear operators and the theory of functions,
V. Zap. Nauc. Sem. Leningrad. Otdel. Mat. Inst. Steklov, (LOMI) 47 (1974), 120--137;
translated in: J. Soviet. Math. 9 (1978), 228--243.

\bibitem{pp1} J. Pau and J. \'A. Pel\'aez,
\textit{Volterra type operators on Bergman spaces with exponential weights},
Topics in complex analysis and operator theory, 239--252,
Contemp. Math., 561, Amer. Math. Soc., Providence, RI, 2012.

\bibitem{pp2} J. Pau and J. \'A. Pel\'aez,
\textit{Schatten classes of integration operators on Dirichlet spaces},
J. Anal. Math. 120 (2013), 255--289.

\bibitem{sz} A. G. Siskakis and R. Zhao,
\textit{A Volterra type operator on spaces of analytic functions},
Function Spaces, Contemp. Math. \textit{232},
Amer. Math. Soc., Providence, RI, 1999, 299--311.

\bibitem{wo} P. Wojtaszczyk,
`Banach Spaces for Analysts',
Cambridge Studies in Advanced Mathematics, 25.
Cambridge University Press, 1991.

\bibitem{x} J. Xiao,
\textit{Riemann-Stieltjes operators on weighted Bloch and Bergman spaces
of the unit ball},
J. London Math. Soc. 70 (2004), 199--214.

\bibitem{zhc} X. Zhang, C. He and F. Cao,
\emph{The Equivalent Norms of $F(p,q,s)$ Space in $\C^n$},
J. Math. Anal. Appl. 401 (2013), 601--610.

\bibitem{zha1} R. Zhao,
 `On a general family of function spaces',
Ann. Acad. Sci. Fenn. Math. Dissertationes
105 (1996), 56 pp.

\bibitem{zha3} R. Zhao,
\textit{New criteria of Carleson measures for Hardy spaces and their applications},
Complex Var. Elliptic Equ. 55 (2010), 633--646.

\bibitem{zz} R. Zhao and K. Zhu,
 `Theory of Bergman spaces in the unit ball of $\C^n$',
Mem. Soc. Math. Fr. (N.S.) 115 (2008), vi+103 pp.

\bibitem{ZhuBn} K. Zhu,
`Spaces of Holomorphic Functions in the Unit Ball',
Springer-Verlag, New York, 2005.



\end{thebibliography}
\end{document}